\documentclass[10pt, oneside]{article}
\usepackage[T1]{fontenc}\usepackage[utf8]{inputenc}
\usepackage{mathrsfs}
\usepackage{amsfonts}
\usepackage{amssymb}
\usepackage{amsmath,amsfonts,amssymb,amsthm}

\usepackage{caption}
\usepackage{subfigure}
\usepackage{cite}
\usepackage{color}
\usepackage{graphicx}
\usepackage{subfigure}
\usepackage{float}
\usepackage{paralist}
\usepackage{indentfirst}
\usepackage{authblk}
\usepackage{graphics,url,color,epsfig}
\usepackage{tikz}
\usetikzlibrary{patterns}
\usepackage[colorlinks]{hyperref}
\AtBeginDocument{
   \hypersetup{
    linkcolor=blue,
    citecolor=blue,
 }
}
\usepackage{cleveref}
\usepackage{pgfplots}
\pgfplotsset{compat=1.18}
\usepackage{color} 

\usepackage[normalem]{ulem}

\usepackage[title]{appendix}

\definecolor{red}{rgb}{1.00,0.00,0.00}
\definecolor{blue}{rgb}{0.00,0.00,0.63}
\definecolor{black}{rgb}{0.00,0.00,0.00}
\definecolor{purple}{rgb}{0.00,1.00,0.00}
\definecolor{pink}{rgb}{0.95,0.01,0.08}

\usepackage[T1]{fontenc}

\newtheorem{theorem}{Theorem}[section]
\newtheorem{lemma}{Lemma}[section]

\newtheorem{cor}[lemma]{Corollary}
\newtheorem{remark}{Remark}[section]

\newtheorem{defn}{Definition}[section]

\newtheorem{exmp}{Example}[section]

\def\bma#1\ema{{\allowdisplaybreaks\begin{aligned}#1\end{aligned}}}

\numberwithin{equation}{section}

\begin{document}

\author{Leyun Wu and Chilin Zhang}
\title{{\LARGE \textbf{Sharp Boundary Growth Rate Estimate of the Singular Equation $-\Delta u=u^{-\gamma}$ in a Critical Cone}}}


\date{}
\maketitle
\begin{abstract}
 For $\gamma>0$, we study the sharp boundary growth rate estimate of solutions to the Dirichlet problem of the singular Lane-Emden-Fowler equation
    \begin{equation*}
        -\Delta u=u^{-\gamma}
    \end{equation*}
    in a critical $C^{1,1}$ epigraphical cone $Cone_{\Sigma}$.
    
We show that the growth rate estimate exhibits fundamentally different behaviors in the following three cases: $1<\gamma<2$, $\gamma=2$, and $\gamma>2$. Moreover, we obtain the sharp growth rate estimate near the origin for $\gamma>1$. As a consequence, we show that when $Cone_{\Sigma}$ is a $C^{1,1}$ epigraphical cone, the additional solvability condition in \cite[Theorem 1.3]{GuLiZh25} is both sufficient and necessary to achieve the growth rate therein, thereby resolving the main open question left in that paper. With the growth rate estimate, we also derive the optimal modulus of continuity for solutions via the interior Schauder estimate.

    Our approach is to control the values of a solution $U(x)$ in the region $\Omega=Cone_{\Sigma}\cap B_{1}$ by introducing  a sequence of reference points $p_{k}=\frac{16^{1-k}}{2}\vec{e_{n}}$. From the Green function representation of $U(x)$, we derive  a discrete integral equation for the sequence $a_{k}=16^{k\phi}U(p_{k})$. Such a computation converts the original PDE problem into a recursion for a discrete integral equation, which can be effectively analyzed  using basic ODE methods.
\end{abstract}

\\
2020 {\bf MSC.} 35B40, 35J75, 45G05

\section{Introduction}
\subsection{Motivation}
Due to the negative power on the right-hand side, the singular Lane-Emden-Fowler equation
\begin{equation}\label{eq. main}
    -\Delta u=u^{-\gamma}
\end{equation}
exhibits singular behavior near the boundary, provided that the boundary data is zero. In order to study the global regularity of the solution, it is necessary first to analyze  the boundary growth rate. Once this is established,  the interior Schauder estimate can be applied to obtain the global H\"older regularity, as well as the precise growth rate of higher-order derivatives near the boundary.

Boundary regularity estimates of such  equations in  smooth domains $\mathcal{D}$ have been extensively studied,  in the literature a long time ago, beginning with the pioneering well-posedness result  by Fulks-Maybee \cite{FuMa60}. Subsequently, a systematic study of \eqref{eq. main} in smooth domains was carried out by Crandall-Rabinowitz-Tartar \cite{CrRaTa77}, as well as by Stuart \cite{St76}. Later, when   $\partial\mathcal{D}\in C^{1,1}$,   Lazer and McKenna \cite{LaMc91} first showed that
\begin{equation*}
    u\sim\phi_{1}^{\frac{2}{1+\gamma}},
\end{equation*}
where $\phi_{1}$ denotes the first eigenfunction of $-\Delta$ under the zero Dirichlet boundary condition. Later in Gui-Lin \cite{GuLi93}, the authors obtained the growth rate estimate for more general $\gamma$'s. They  showed that
\begin{equation}\label{eq. Gui-Lin estimate}
    u(x)\sim\left\{\begin{aligned}
        &\left(dist(x,\partial\mathcal{D})\right)^{\frac{2}{1+\gamma}},& \ \mbox{if}\ \gamma>1,\\ 
        &dist(x,\partial\mathcal{D})|\ln{dist(x,\partial\mathcal{D})}|^{\frac{1}{2}} ,& \ \mbox{if}\ \gamma=1,\\
        &dist(x,\partial\mathcal{D}),& \ \mbox{if}\ \gamma<1.
    \end{aligned}\right.
\end{equation}
In fact, the result in \cite{GuLi93} covers a more general class of singular equations with a degenerate weight:
\begin{equation*}
    -\Delta u=f(x)\cdot u^{-\gamma},\quad f\sim \left(dist(x,\partial\mathcal{D})\right)^{\alpha}.
\end{equation*}
See also del Pino \cite{dP92} and Gomes \cite{Go86} for related results on weighted singular equations.

The energy method is also widely used in the study of singular equations with an $L^{p}$ weight, that is, 
\begin{equation*}
    -\Delta u=f(x)\cdot u^{-\gamma},\quad f\in L^{p}(\mathcal{D}).
\end{equation*}
For such equations,  Boccardo-Orsina \cite{BoOr10} and Oliva-Petitta \cite{OlPe18} showed that the condition $\gamma<3-\frac{2}{p}$ is equivalent to the existence of a $H^{1}$ solution with zero Dirichlet boundary data.

When the domain is the half-space $\mathbb{R}^{n}_{+}$,  several interesting classification results for global solutions are known. For example, global solutions to \eqref{eq. main} in $\mathbb{R}^{n}_{+}$ that vanish on $\partial\mathbb{R}^{n}_{+}$ were classified by Montoro-Muglia-Sciunzi \cite{MoMuSc24a,MoMuSc24b}. A similar classification result for the fractional equation $(-\Delta)^{s}u=u^{-\gamma}$ was obtained by Guo-Wu \cite{GuWu25}.

If $\partial\mathcal{D}$ is less regular and possesses a $\frac{\pi}{2}$ angle, then the growth rate estimate and a Krylov-type boundary Harnack principle of \eqref{eq. main} was derived by Huang-Zhang \cite{HuZh24}. In a recent work Guo-Li-Zhang \cite{GuLiZh25}, the authors systematically studied the boundary regularity of solutions to \eqref{eq. main} in  general Lipschitz domains. Analogous to the interior and exterior ball method in \cite{GuLi93}, the authors in \cite{GuLiZh25} considered the limiting cone:
\begin{equation*}
    Cone_{\Sigma}=\Big\{x\in\mathbb{R}^{n}\setminus\{0\}:\frac{x}{|x|}\in\Sigma\Big\}\quad(\mbox{here }\Sigma\subseteq\mathbb{S}^{n-1}=\partial B_{1})
\end{equation*}
centered at a boundary point (without loss of generality, the origin).  The pairs $(\Sigma,\gamma)$ were classified into three categories and were treated accordingly. More precisely, they first define the ``frequency'' of a cone:
\begin{defn}[The ``frequency'' of a cone]\label{def. frequency}
    Let $\lambda=\lambda_{\Sigma}$ be the first eigenvalue of the Laplace-Beltrami operator $-\Delta_{\Sigma}$ in $\Sigma$ and let $E(\vec{\theta})\geq0$ be a corresponding eigenfunction. Define $\phi=\phi_{\Sigma}$ to be the positive solution to $\phi(\phi+n-2)=\lambda$, and we call $\phi$ the ``frequency'' of $Cone_{\Sigma}$. In fact, there exists a homogeneous positive harmonic function
    \begin{equation}\label{eq. H Sigma, homogeneous harmonic in a cone}
        H(x)=H_{\Sigma}(x)=|x|^{\phi}E(\frac{x}{|x|})
    \end{equation}
    supported in the cone $Cone_{\Sigma}$.
\end{defn}
Based on this notion, the pair  $(\Sigma,\gamma)$ is classified as:
\begin{itemize}
    \item subcritical, if $\frac{2}{1+\gamma}<\phi$;
    \item critical, if $\frac{2}{1+\gamma}=\phi$;
    \item supercritical, if $\frac{2}{1+\gamma}>\phi$.
\end{itemize}
\begin{exmp}[The critical cone in $\mathbb{R}^{2}$]\label{ex. 2d how large is the critical angle}
    It is easy to see that in $\mathbb{R}^{2}$, if
    \begin{equation*}
        \Sigma=\Big\{\sin{\theta}\vec{e_{1}}+\cos{\theta}\vec{e_{2}}:\theta\in[-\frac{1+\gamma}{4}\pi,\frac{1+\gamma}{4}\pi]\Big\}\subseteq\mathbb{S}^{1}=\partial B_{1},
    \end{equation*}
    then the pair $(\Sigma,\gamma)$ is critical, and $Cone_{\Sigma}$ is an angle in $\mathbb{R}^{2}$ of size $\frac{1+\gamma}{2}\pi$.
\end{exmp}
\begin{exmp}
    Moreover, we can check that if an angle is less than $\frac{1+\gamma}{2}\pi$, then the pair $(\Sigma,\gamma)$ is subcritical. If the angle is greater than $\frac{1+\gamma}{2}\pi$, then the pair $(\Sigma,\gamma)$ is supercritical.
\end{exmp}

If $\partial Cone_{\Sigma}$ is a Lipschitz graph, then the boundary growth rate of a solution to \eqref{eq. main} in $B_{1}\cap\partial Cone_{\Sigma}$ vanishing on $\partial Cone_{\Sigma}$ was established  in \cite[Theorem 1.2]{GuLiZh25}. Precisely, in $B_{1/2}\cap Cone_{\Sigma}$, we have 
\begin{equation*}
u(x)\leq 
\begin{cases}
    C|x|^{\frac{2}{1+\gamma}},&\quad\mbox{in the subcritical case},\\
    C|x|^{\phi}\ln(\frac{1}{|x|}),&\quad\mbox{in the critical case},\\
     C|x|^{\phi},&\quad\mbox{in the supercritical case}.
    \end{cases}
\end{equation*}
The exponents $\frac{2}{1+\gamma}$ in the subcritical case and $\phi$ in the supercritical case have been shown to be optimal, and solutions exhibit good behavior in these two cases. For instance, the exponent $\frac{2}{1+\gamma}$ is stable in the subcritical case,  does not depend on the precise geometry of  $Cone_{\Sigma}$. In the supercritical case, the exponent $\phi$ is less stable, but under an additional ``interior cone condition'', the solution can be shown to be  divisible by a harmonic function in a natural sense. 

However, the boundary behavior of a solution is more intricate in the critical case. For example, there is a gap in the growth rate in the critical case, since constructing a matching precise lower barrier is subtle and more involved.  In \cite{GuLiZh25}, a lower barrier of the form
\begin{equation}\label{eq. precise lower solution in GLZ}
    \underline{U}(x)=K(\ln{\frac{1}{|x|}})^{\phi/2}H_{\Sigma}(x)
\end{equation}
was constructed for a sufficiently small $K$ was constructed, implying that 
\begin{equation}\label{eq. lower bound general}
    u(t\vec{e_{n}})\geq C^{-1}t^{\phi}(\ln\frac{1}{t})^{\phi/2}.
\end{equation}

Then a natural question arises: ``What is the optimal growth rate of $u(x)$ if the pair $(\Sigma,\gamma)$ is critical?'' A partial answer was given in \cite[Theorem 1.3]{GuLiZh25} under an additional solvability assumption, namely,  that there exists a $w\in C(\overline{B_{1}\cap Cone_{\Sigma}})$ satisfying:
\begin{equation}\label{eq. growth rate special condition in critical case}
   \begin{cases}
        -\Delta w=H_{\Sigma}^{-\gamma}&\mbox{ in }\,\,Cone_{\Sigma}\cap B_{1},\\
        w=0&\mbox{ on }\,\,\partial(Cone_{\Sigma}\cap B_{1}).
    \end{cases}
\end{equation}
Under this solvability condition, one obtains the sharp upper bound
\begin{equation}\label{eq. upper bound with solvability assumption}
    u(x)\leq C|x|^{\phi}(\ln\frac{1}{|x|})^{\phi/2},
\end{equation}
which perfectly matches the lower bound estimate \eqref{eq. lower bound general}. Without such an assumption, the optimal growth rate remained unclear—this is precisely the main question addressed in the present paper.

In fact, the authors in \cite{GuLiZh25} suggested that one could first study the growth rate in an angle in $\mathbb{R}^{2}$, and they conjectured that \eqref{eq. upper bound with solvability assumption} holds if and only if $\gamma<2$ (equivalently, the angle is less than $\frac{3\pi}{2}$). In this paper, we are able to obtain the optimal growth rate of the solution not only for  a two-dimensional angle, but also under the assumption that $Cone_{\Sigma}$ is a $C^{1,1}$ epigraphical cone in $\mathbb{R}^{n}$ for any dimension $n\geq2$. As will be stated in Theorem~\ref{thm. three equivalent statements}, the solvability condition \eqref{eq. growth rate special condition in critical case} is in fact necessary to achieve the estimate \eqref{eq. upper bound with solvability assumption}, provided that $Cone_{\Sigma}$ is smooth away from the origin.

\subsection{Main results}
Now we state our main results. First, we set $f(x)\equiv1$ and study the simplified equation \eqref{eq. main}. We focus on a special solution $U(x)$ of the problem:
\begin{equation}\label{eq. special solution U(x)}
    \begin{cases}
        -\Delta U=U^{-\gamma}&\mbox{in }\,\,\Omega=Cone_{\Sigma}\cap B_{1},\\
        U=0&\mbox{on }\,\,\partial\Omega.
    \end{cases}
\end{equation}
Even for such a special solution $U(x)$, its growth rate near the boundary is highly non-trivial and extremely counterintuitive. We present below a full statement of the growth rate estimate of $U(x)$ everywhere in $\Omega$ for all $\gamma>1$.
\begin{theorem}\label{thm. asymptotic estimate}
    Assume that $Cone_{\Sigma}$ is a $C^{1,1}$ epigraphical cone with norms $(L,K)$ (see Definition~\ref{def. a good cone}) such that its ``frequency'' $\phi=\phi_{\Sigma}$ (see Definition~\ref{def. frequency}) satisfies the critical cone assumption $\phi=\frac{2}{1+\gamma}$, and let $U(x)$ be the solution of \eqref{eq. special solution U(x)}. We fix a defining function $\eta(x)$ as in Definition~\ref{def. defining function eta}, and denote $\Omega=Cone_{\Sigma}\cap B_{1}$. Then, we have the following sharp asymptotic estimates in three cases:
    \begin{itemize}
        \item[(1)] If $1<\gamma<2$, we call this situation ``critical-subcritical'', and
        \begin{equation*}
            U(x)\sim\left\{\begin{aligned}
                 &\left(dist(x,\partial\Omega)\right)^{\frac{2}{1+\gamma}}&\mbox{in }&\Omega\setminus B_{1/16},\\
                 &|x|^{\frac{2}{1+\gamma}}\eta^{\frac{2}{1+\gamma}}(x)\sim \left(dist(x,\partial\Omega)\right)^{\frac{2}{1+\gamma}}&\mbox{in }&\Omega\cap B_{1/16}\cap\Big\{\eta(x)\leq(\ln{\frac{1}{|x|}})^{\frac{1}{1-\gamma}}\Big\},\\
                 &|x|^{\frac{2}{1+\gamma}}(\ln{\frac{1}{|x|}})^{\frac{1}{1+\gamma}}\eta(x)&\mbox{in }&\Omega\cap B_{1/16}\cap\Big\{\eta(x)\geq(\ln{\frac{1}{|x|}})^{\frac{1}{1-\gamma}}\Big\}.
            \end{aligned}\right.
        \end{equation*}
        \item[(2)] If $\gamma=2$, we call this situation ``critical-critical'', and
        \begin{equation*}
            U(x)\sim\left\{\begin{aligned}
                 &\left(dist(x,\partial\Omega)\right)^{\frac{2}{3}}&\mbox{in }&\Omega\setminus B_{1/16},\\
                 &|x|^{\frac{2}{3}}\eta^{\frac{2}{3}}(x)\sim \left(dist(x,\partial\Omega)\right)^{\frac{2}{3}}&\mbox{in }&\Omega\cap B_{1/16}\cap\Big\{\eta(x)\leq\big[(\ln{\frac{1}{|x|}})\cdot(\ln{\ln{\frac{1}{|x|}}})\big]^{-1}\Big\},\\
                 &|x|^{\frac{2}{3}}\big[(\ln{\frac{1}{|x|}})\cdot(\ln{\ln{\frac{1}{|x|}}})\big]^{\frac{1}{3}}\eta(x)&\mbox{in }&\Omega\cap B_{1/16}\cap\Big\{\eta(x)\geq\big[(\ln{\frac{1}{|x|}})\cdot(\ln{\ln{\frac{1}{|x|}}})\big]^{-1}\Big\}.
            \end{aligned}\right.
        \end{equation*}
        \item[(3)] If $\gamma>2$, we call this situation ``critical-supercritical'', and
        \begin{equation*}
            U(x)\sim\left\{\begin{aligned}
                 &\left(dist(x,\partial\Omega)\right)^{\frac{2}{1+\gamma}}&\mbox{in }&\Omega\setminus B_{1/16},\\
                 &|x|^{\frac{2}{1+\gamma}}\eta^{\frac{2}{1+\gamma}}(x)\sim \left(dist(x,\partial\Omega)\right)^{\frac{2}{1+\gamma}}&\mbox{in }&\Omega\cap B_{1/16}\cap\Big\{\eta(x)\leq(\ln{\frac{1}{|x|}})^{-1}\Big\},\\
                 &|x|^{\frac{2}{1+\gamma}}(\ln{\frac{1}{|x|}})^{\frac{\gamma-1}{1+\gamma}}\eta(x)&\mbox{in }&\Omega\cap B_{1/16}\cap\Big\{\eta(x)\geq(\ln{\frac{1}{|x|}})^{-1}\Big\}.
            \end{aligned}\right.
        \end{equation*}
    \end{itemize}
    Here, we use ``$\sim$'' to denote equivalence, see Definition~\ref{def. convention of uniform and equivalence} for the precise conventions.
\end{theorem}
\begin{remark}
    Figure 1 below illustrates the domain decomposition described in Theorem~\ref{thm. asymptotic estimate}. The first line of each case corresponds to the outer annulus, i.e., the region between the larger and smaller arcs. The second line of each case corresponds to the shadowed region. The third line of each case corresponds to the unshadowed region contained within the smaller arc.
\end{remark}

\begin{center}
\begin{tikzpicture}[scale=0.3][node distance = 0.1cm]
\draw[black,line width=1.0] (5.656,-5.656) arc (-45:225:8);
 \draw [thick] (-5.656,-5.656) -- (0,0);
  \draw [thick] (5.656,-5.656) -- (0,0);
    \draw[blue,line width=1.0] (3.182,-3.182) arc (-45:225:4.5);
        \draw[red,line width=1.2] [dashed](-4,-2) arc (280:315:7.5);
        \draw[red,line width=1.2] [dashed](4,-2) arc (260:225:7.5);
         \node [below=1cm, align=flush center,text width=8cm] at (0,-3.5)
        {$Figure$ 1 \fontsize{15}{15}\selectfont };
        \def\mypath{(0,0) -- (-3.182, -3.182) arc(225:210:4.5) -- (-4,-2) arc (280:315:7.5)}
       \pattern[pattern color=gray!80, pattern=north west lines]  \mypath;
          \def\mypath{(4,-2) arc (260:225:7.5) -- (0, 0)-- (3.182,-3.182)arc(315:330:4.5) }
       \pattern[pattern color=gray!80, pattern=north west lines]  \mypath;
\end{tikzpicture}
\end{center}

\begin{remark}
    The main differences among these three cases lie in the domain decomposition and  the power of $\ln{\frac{1}{|x|}}$. In the intermediate case $\gamma=2$, the situation becomes more intricate, as a dual-logarithmic term $\ln{\ln{\frac{1}{|x|}}}$ appears both in the estimates and in the domain decomposition.
\end{remark}

Next, we consider a general solution $u(x)$ in $\Omega$ and  establish a continuity estimate in $\Omega\cap B_{1/16}$. Our strategy is to compare $u(x)$ with the special solution $U(x)$. Once the growth rate of $u(x)$ is obtained, a straightforward application of the interior Schauder estimate yields the modulus of continuity.

\begin{theorem}\label{thm. modulus of continuity}
    Let $\gamma>1$. Assume that $Cone_{\Sigma}$ is a $C^{1,1}$ epigraphical cone with norms $(L,K)$ such that its ``frequency'' satisfies $\phi=\frac{2}{1+\gamma}$. Denote $\Omega=Cone_{\Sigma}\cap B_{1}$ and let $u(x)$ be a solution to
    \begin{equation}\label{eq. too general main equation}
        -\Delta u(x)=f(x)\cdot u^{-\gamma}(x)\quad\mbox{with }0<\lambda\leq f(x)\leq\Lambda
    \end{equation}
    in $\Omega$ which vanishes on $\partial Cone_{\Sigma}$. Then, there exists a constant $C=C(n,\gamma,\lambda,\Lambda,L,K)$, such that we have the following growth rate and continuity estimates in $\Omega\cap B_{1/16}$:
    \begin{itemize}
        \item[(1)] Let $U(x)$ be the special solution to \eqref{eq. special solution U(x)}, then in $\Omega\cap B_{1/16}$,
        \begin{equation*}
            C^{-1}\cdot U(x)\leq u(x)\leq C\|u\|_{L^{\infty}(\Omega)}\cdot U(x).
        \end{equation*}
        See estimate of $U(x)$ in Theorem~\ref{thm. asymptotic estimate}.
        \item[(2)] Let $\sigma(t)$ be defined for $t\leq1/8$ as:
        \begin{equation*}
            \sigma(t)=\left\{\begin{aligned}
                &t^{\frac{2}{1+\gamma}}(\ln{\frac{1}{t}})^{\frac{1}{1+\gamma}},&\mbox{if }&1<\gamma<2,\\
                &t^{\frac{2}{3}}(\ln{\frac{1}{t}})^{\frac{1}{3}}(\ln{\ln{\frac{1}{t}}})^{\frac{1}{3}},&\mbox{if }&\gamma=2,\\
                &t^{\frac{2}{1+\gamma}}(\ln{\frac{1}{t}})^{\frac{\gamma-1}{1+\gamma}},&\mbox{if }&\gamma>2.
            \end{aligned}\right.
        \end{equation*}
        Then, we have $u(x)\in C^{\sigma}(\Omega\cap B_{1/16})$, in the sense that for every $x,y\in\Omega\cap B_{1/16}$,
        \begin{equation*}
            |u(x)-u(y)|\leq C\|u\|_{L^{\infty}(\Omega)}\cdot\sigma(|x-y|).
        \end{equation*}
        Such a modulus of continuity is optimal, in the sense that if $u(x)\in C^{\widetilde{\sigma}}(\Omega\cap B_{1/16})$ for some other modulus of continuity $\widetilde{\sigma}$, then
        \begin{equation*}
            \sup_{t\leq1/8}\frac{\sigma(t)}{\widetilde{\sigma}(t)}<\infty.
        \end{equation*}
    \end{itemize}
\end{theorem}
\begin{remark}\label{rmk. weird estimate why?}
   It may seem unusual that in Theorem~\ref{thm. modulus of continuity} (1) we have the upper bound  $u(x)\leq C\|u\|_{L^{\infty}(\Omega)}\cdot U(x)$, but only $u(x)\geq C^{-1} U(x)$. It is important to emphasize that the seemingly ``natural'' estimate
   $u(x)\geq C^{-1}\|u\|_{L^{\infty}(\Omega)}\cdot U(x)$ is actually incorrect. 
   The reason for this subtle point will be explained in Remark~\ref{rmk. weird estimate reason} following Theorem~\ref{thm. ratio tends to 1}.
\end{remark}

Now we can address the open question left in \cite{GuLiZh25}. Specifically, if the cone $Cone_\Sigma$ is smooth away from the origin, then we show that \eqref{eq. growth rate special condition in critical case} is both necessary and sufficient for the estimate \eqref{eq. upper bound with solvability assumption} to hold.

\begin{theorem}\label{thm. three equivalent statements}
    Let $\gamma>0$. Assume that $Cone_{\Sigma}$ is a $C^{1,1}$ epigraphical cone with norms $(L,K)$ such that its ``frequency'' satisfies $\phi=\frac{2}{1+\gamma}$. Denote $\Omega=Cone_{\Sigma}\cap B_{1}$ and let $u(x)$ be a solution to \eqref{eq. too general main equation} in $\Omega$ which vanishes on $\partial Cone_{\Sigma}$. Then, the following three statements are equivalent:
    \begin{itemize}
        \item[(a)] $\gamma<2$;
        \item[(b)] The solvability condition \eqref{eq. growth rate special condition in critical case} holds;
        \item[(c)] There exists some constant $C=C(n,\gamma,L,K,\lambda,\Lambda)$ such that for all $r\leq1/2$,
        \begin{equation*}
            C^{-1}\cdot r^{\phi}(\ln{\frac{1}{r}})^{\phi/2}\leq\max_{Cone_{\Sigma}\cap B_{r}}u\leq C\|u\|_{L^{\infty}(\Omega)}\cdot r^{\phi}(\ln{\frac{1}{r}})^{\phi/2}.
        \end{equation*}
    \end{itemize}
\end{theorem}

Finally, for two different solutions $u,v$ satisfying \eqref{eq. too general main equation} in $\Omega$ and vanishing on $\partial Cone_{\Sigma}$, we would like to analyze the behavior of the ratio $u/v$. By \cite[Theorem 1.4]{GuLiZh25} (or see Lemma~\ref{lem. Harnack Comparable, general Lipschitz domain}), it is  already known that such a ratio is bounded above and below in $\Omega\cap B_{1/16}$. However, boundedness of the ratio $u/v$ does not naturally imply its continuity up to the boundary, see an example in \cite[Theorem 1.5]{GuLiZh25}. This is a key difference from the classical boundary Harnack principle for linear equations. Therefore, in this paper, we also investigate the continuity of the ratio $u/v$, at least in the special case for the conical domain $\Omega$.

\begin{theorem}\label{thm. ratio tends to 1}
    Let $\gamma>1$ and let $u(x)$, $v(x)$ be two distinct solutions of \eqref{eq. too general main equation} in $\Omega$, both vanishing on $\partial Cone_{\Sigma}$ as in Theorem~\ref{thm. three equivalent statements}. Then in $\overline{\Omega\cap B_{1/16}}$, the ratio $\frac{u(x)}{v(x)}$ is continuous, bounded away from $0$ and $\infty$, and its limit at $B_{1/16}\cap\partial Cone_{\Sigma}$ is $1$ everywhere.
\end{theorem}
\begin{remark}\label{rmk. weird estimate reason}
   We now clarify why the estimate  $u(x)\geq C^{-1}\|u\|_{L^{\infty}(\Omega)}\cdot U(x)$ is incorrect, as mentioned  in Remark~\ref{rmk. weird estimate why?}. To see this, we fix any $v(x)$ as in Theorem~\ref{thm. ratio tends to 1}, and obtain that $u(x)/v(x)\to1$ as $x\to 0$. It then follows that 
    \begin{equation*}
        u(x)\leq2v(x)\leq C\|v\|_{L^{\infty}(\Omega)}\cdot U(x)\mbox{ when }|x|\mbox{ is small}.
    \end{equation*}
    Therefore,  one should not expect that  $u(x)\geq C^{-1}\|u\|_{L^{\infty}(\Omega)}\cdot U(x)$ in $\Omega\cap B_{1/16}$, especially when $\|u\|_{L^{\infty}(\Omega)}$ is large.
\end{remark}

\subsection{Key ideas}
Less elaborate computation does yield some optimal growth rate estimates, but only under the solvability condition \eqref{eq. growth rate special condition in critical case}. Therefore, to obtain the optimal growth rate estimate in more general 
settings, the analysis must be much more precise. To this end, our strategy  is to study the integral form of \eqref{eq. main}, namely,
\begin{equation}\label{eq. key integral equation using Green function}
    U(x)=\int_{\Omega}G(x,y)U^{-\gamma}(y) dy,
\end{equation}
where  $G(x,y)$  is the positive Green  function of $\Omega$ for the operator $-\Delta$, vanishes on $\partial\Omega$, and satisfies 
\begin{equation}\label{eq. green function (intro)}
    G(x,y)=(-\Delta_{y})^{-1}\delta(x-y).
\end{equation}

There are two key difficulties in applying the integral equation \eqref{eq. key integral equation using Green function}:
\begin{itemize}
    \item First, unlike well-behaved domains such as the  ball $B_{1}$, the half space $\mathbb{R}^{n}_{+}$ or the entire space $\mathbb{R}^{n}$, the Green function in a conical region $\Omega=Cone_{\Sigma}\cap B_{1}$ can not be written explicitly, and it is not linearly growing near $\partial\Omega$.
    \item Second, it is difficult to determine in advance the accurate growth rate of $U(y)^{-\gamma}$ inside the  integral, while rough estimates of $U(y)$ alone are insufficient to obtain  the optimal estimate for the integral.
\end{itemize}

In Section~\ref{sec. green function} and  Section~\ref{sec. estimates in annuli}, we focus on overcoming these two main difficulties. To summarize in large, our approach is to ``discretize'' the integral equation problem \eqref{eq. key integral equation using Green function}. We select a ``geometric series'' of reference points, and treat the values of $U(x)$ at these points as ``unknowns''. Using these discrete ``unknowns'', we can estimate $U(y)^{-\gamma}$ in the most accurate way. Then in Section~\ref{sec. integral equation}, 
we estimate the integral \eqref{eq. key integral equation using Green function} by splitting the domain into a sequence of annuli, effectively turning the integral into a discrete sum.
 With this  ``discretization'', it turns out that to a certain extent, we analyze the equation \eqref{eq. key integral equation using Green function} by ``directly solving it''.

The method of analyzing the integral equation \eqref{eq. key integral equation using Green function} in this paper also applies when the pair $(\Sigma,\gamma)$ is subcritical or supercritical, as well as for $\gamma\leq1$.  We believe that  some interesting and counterintuitive results can also be obtained in these cases.

What follows is the organization of this paper:

In Section~\ref{sec. green function}, we address the first difficulty by carefully estimating the Green function \eqref{eq. green function (intro)}, obtaining  sharp bounds throughout $Cone_{\Sigma}$. Our approach is inspired by Bogdan \cite{Bo00}, see also Duraj-Wachtel \cite{DuWa18}. In particular, we apply Kemper's boundary Harnack principle \cite{Ke72} (or see Lemma~\ref{lem. classical BHP})  to study the growth rate of the Green function near the boundary.

In Section~\ref{sec. estimates in annuli}, we deal with the second difficulty. Precisely, we choose the  reference points $p_{k}=\frac{16^{1-k}}{2}\vec{e_{n}}$, and define a sequence of ``unknowns'' by
\begin{equation*}
    a_{k}=16^{k\phi}U(p_{k}).
\end{equation*}
Using the estimate from \cite[Theorem 1.4]{GuLiZh25} and the one-dimensional solutions of \eqref{eq. main} obtained in \cite{MoMuSc24a,MoMuSc24b}, we analyze  the value at each point in $\Omega$, and represent it in terms of  piecewise algebraic expressions involving the quantities $a_{k}$'s, see Corollary~\ref{cor. estimate in a cone}. Precisely speaking, we divide the bounded conic region $\Omega$ into annuli $\mathbb{A}_{k}$, and obtain asymptotic estimates in each annulus $\mathbb{A}_{k}$ using the ``unknown'' $a_{k}$. It is worth  noting that in this section, we also derive an improved and localized version of Gui-Lin's estimate \eqref{eq. Gui-Lin estimate} in the case $\gamma>1$, see Theorem~\ref{thm. by-product}. 
This refinement provides better control at a microscopic scale—particularly when the 
 $L^{\infty}$ norm of $u$ or the domain $\mathcal{D}$ is large,  our estimate in Theorem~\ref{thm. by-product} yields  more accurate behavior near the boundary.

In Section~\ref{sec. integral equation}, we set $x=p_{k}$ and compute $U(p_{k})$ using the integral equation \eqref{eq. key integral equation using Green function}. The computation of the right-hand side is delicate, as it requires estimating the integral over each annulus $\mathbb{A}_{k}$ and then adding them up. By combining the sharp Green function  estimates and the ``$a_{k}$-related estimate'' of $U(y)$, we obtain a discrete integral equation, such that the value $a_{k}$ is equivalent to some algebraic expression of the sequence $\{a_{k}\}_{k\geq1}$. We then obtain the sharp growth rate estimate of $a_{k}$ from the discrete integral equation.

The main results are then proven in Section~\ref{sec. proof of main results}. 
 For clarity, we also include Appendix~\ref{sec. appendix: notations} and Appendix~\ref{sec. appendix: review}, which summarize our notational conventions and collect several useful background results.

\section{Sharp estimates for the Green function \texorpdfstring{$G(x,y)$}{Lg}}\label{sec. green function}
We define the Green function for the operator $-\Delta$ in the domain $\Omega$ by $G(x,y)$, which satisfies the following equation in the distributional sense:
\begin{equation*}
    \begin{cases}
        -\Delta_{y}G(x,y)=\delta(x-y)&\mbox{in }\,\,\Omega,\\
        G(x,y)=0&\mbox{on }\,\,\partial\Omega.
    \end{cases}
\end{equation*}
Notice that $G(x,y)$ must be positive (including taking the value $+\infty$) in $\Omega$. Besides, it is a well-known fact that $G(x,y)=G(y,x)$.

In this section, we investigate the asymptotic behavior of 
 $G(x,y)$ near the boundary and near the pole, given that the pole $x$ lies on the positive $x_{n}$-axis. The estimates presented below (Lemma~\ref{lem. green function for small y} and Lemma~\ref{lem. green function for large y}) are sharp and are stated  in the equivalence sense ``$\sim$'' as defined in Definition~\ref{def. convention of uniform and equivalence}.

\begin{lemma}[Green function for small $y$]\label{lem. green function for small y}
    Let $x=h\vec{e_{n}}$ with $h\leq\frac{1}{4}$, and
    \begin{equation*}
        y=r\vec{\theta}.
    \end{equation*}
    with $\vec{\theta}\in\Sigma\subseteq\partial B_{1}$ and $r\leq2h$. Let $\eta(x)$ be the defining function as in Definition~\ref{def. defining function eta}. Then the Green function $G(x,y)$ satisfies:
    \begin{equation*}
        G(x,y)\sim\left\{\begin{aligned}
            &\ln{\frac{8h}{|x-y|}}\cdot (\frac{r}{h})^{\phi}\cdot\eta(y),&\mbox{if }n=2,\\
            &|x-y|^{2-n}(\frac{r}{h})^{\phi}\cdot\eta(y),&\mbox{if }n\geq3.
        \end{aligned}\right.
    \end{equation*}
\end{lemma}
\begin{proof}
    Let's decompose
    \begin{equation*}
        G(x,y)=\left\{\begin{aligned}
            &\frac{1}{2\pi}\ln{\frac{1}{|x-y|}}-\widetilde{G}(x,y),&\mbox{if }n=2,\\
            &\frac{|x-y|^{2-n}}{n(n-2)\omega_{n}}-\widetilde{G}(x,y),&\mbox{if }n\geq3,
        \end{aligned}\right.
    \end{equation*}
    where the compensation function $\widetilde{G}(x,y)$ satisfies
    \begin{equation*}
        -\Delta_{y}\widetilde{G}(x,y)=0\mbox{ in }\Omega,\quad\widetilde{G}(x,y)=\left\{\begin{aligned}
        &\frac{1}{2\pi}\ln{\frac{1}{|x-y|}},&\mbox{if }&n=2,\\
        &\frac{|x-y|^{2-n}}{n(n-2)\omega_{n}},&\mbox{if }&n\geq3,
    \end{aligned}\right.
    \end{equation*}
    for any $y \in \partial \Omega$.
    \begin{itemize}
        \item Step 1: Barriers for $\widetilde{G}(x,y)$. 
        
        Let $x^{*}$ be the ``reflection'' of $x$ on the negative $\vec{e_{1}}$ axis, defined by
        \begin{equation*}
            x^{*}=-h\vec{e_{1}}.
        \end{equation*}
         We claim that for any  $y\in\partial\Omega$, it holds that
        \begin{equation*}
            |x^{*}-y|\sim|x-y|.
        \end{equation*}
      To verify this,  we consider two cases. 
      
      {\it{Case 1.}} If $y\in Cone_{\Sigma}\cap\partial B_{1}$, then by the triangle inequality, we have
        \begin{equation*}
            \frac{3}{4}\leq1-h\leq|x^{*}-y|,\,\,|x-y|\leq1+h\leq\frac{5}{4}.
        \end{equation*}
        
       {\it{Case 2.}}  If $y\in B_{1}\cap \partial Cone_{\Sigma}$, we write
        \begin{equation*}
            y=r\vec{\theta},\quad\mbox{where }\vec{\theta}\in\partial\Sigma.
        \end{equation*}
        Since  $\partial Cone_{\Sigma}$ is a Lipschitz graph with norm $L$, we  have
        \begin{equation*}
            -\frac{L}{\sqrt{L^{2}+1}}\leq\cos{\angle(\vec{\theta},\vec{e_{n}})}\leq\frac{L}{\sqrt{L^{2}+1}},\quad\mbox{for all }\vec{\theta}\in\partial\Sigma.
        \end{equation*}
        By the cosine law, one has 
        \begin{equation*}
            \frac{|x^{*}-y|^{2}}{|x-y|^{2}}=\frac{r^{2}+h^{2}+2rh\cos{\angle(\vec{\theta},\vec{e_{n}})}}{r^{2}+h^{2}-2rh\cos{\angle(\vec{\theta},\vec{e_{n}})}}=:\frac{1+t\cos{\angle(\vec{\theta},\vec{e_{n}})}}{1-t\cos{\angle(\vec{\theta},\vec{e_{n}})}},
        \end{equation*}
        where for $0\leq r\leq1$,
        \begin{equation*}
            t:=\frac{2rh}{r^{2}+h^{2}}\in[0,1].
        \end{equation*}
        Using the bounds on  $\cos{\angle(\vec{\theta},\vec{e_{n}})}$, we see that there exists a universal constant $C >0$ such that
        \begin{equation}\label{eq. compensation bound 1}
            C^{-1}\leq\frac{|x^{*}-y|}{|x-y|}\leq C,\quad\mbox{for all }y\in\partial\Omega\mbox{ and }x=h\vec{e_{1}}\mbox{ with }0<h\leq\frac{1}{4}.
        \end{equation}
        With \eqref{eq. compensation bound 1}, we can find suitable harmonic upper and lower bounds for $\widetilde{G}(x,y)$. When $n=2$, we have
        \begin{equation}\label{eq. compensation bound 2}
            \frac{1}{2\pi}\ln{\frac{1}{|x_{*}-y|}}-C\leq\widetilde{G}(x,y)\leq\frac{1}{2\pi}\ln{\frac{1}{|x_{*}-y|}}+C,\quad\mbox{for some uniform }C.
        \end{equation}
        When $n\geq3$, we have
        \begin{equation}\label{eq. compensation bound 2, higher dimension}
            C^{-1}|x_{*}-y|^{2-n}\leq\widetilde{G}(x,y)\leq C|x_{*}-y|^{2-n},\quad\mbox{for some uniform }C.
        \end{equation}
        \item Step 2: Estimate of $G(x,y)$ near the pole. We consider a point $y\in\partial B_{\lambda h}(x)$, where $\lambda$ is sufficiently small. In this case, we have by \eqref{eq. compensation bound 2} or \eqref{eq. compensation bound 2, higher dimension} that
        \begin{equation*}
            \left\{\begin{aligned}
                &\frac{1}{2\pi}\ln{\frac{1}{(2+\lambda)h}}-C\leq\widetilde{G}(x,y)\leq\frac{1}{2\pi}\ln{\frac{1}{(2-\lambda)h}}+C,&\mbox{if }n=2,\\
                &C^{-1}\frac{\Big((2+\lambda)h\Big)^{2-n}}{n(n-2)\omega_{n}}\leq\widetilde{G}(x,y)\leq C\frac{\Big((2-\lambda)h\Big)^{2-n}}{n(n-2)\omega_{n}},&\mbox{if }n\geq3.
            \end{aligned}\right.
        \end{equation*}
        By choosing a larger constant $C$ (still uniform), we have for $y\in\partial B_{\lambda h}(x)$:
        \begin{equation*}
            \left\{\begin{aligned}
                &\frac{1}{2\pi}\ln{\frac{1}{2h}}-C\leq\widetilde{G}(x,y)\leq\frac{1}{2\pi}\ln{\frac{1}{2h}}+C,&\mbox{if }n=2,\\
                &C^{-1}\frac{(2h)^{2-n}}{n(n-2)\omega_{n}}\leq\widetilde{G}(x,y)\leq C\frac{(2h)^{2-n}}{n(n-2)\omega_{n}},&\mbox{if }n\geq3.
            \end{aligned}\right.
        \end{equation*}
        Since in the situation $y\in\partial B_{\lambda h}(x)$, one has
        \begin{equation*}
        \left\{
        \begin{aligned}
            &\frac{1}{2\pi}\ln{\frac{1}{|x-y|}}=\frac{1}{2\pi}\ln{\frac{1}{\lambda h}}&\mbox{for }n=2,\\
            &\frac{|x-y|^{2-n}}{n(n-2)\omega_{n}}=\frac{(\lambda h)^{2-n}}{n(n-2)\omega_{n}}&\mbox{for }n\geq3,
        \end{aligned}\right.
        \end{equation*}
        it holds for any $y\in\partial B_{\lambda h}(x)$ that 
        \begin{equation*}
            \left\{\begin{aligned}
                &\frac{1}{2\pi}\ln{\frac{2}{\lambda}}-C\leq G(x,y)\leq\frac{1}{2\pi}\ln{\frac{2}{\lambda}}+C,&\mbox{if }n=2,\\
                &C^{-1}\frac{(\lambda h)^{2-n}}{n(n-2)\omega_{n}}\leq G(x,y)\leq C\frac{(\lambda h)^{2-n}}{n(n-2)\omega_{n}},&\mbox{if }n\geq3.
            \end{aligned}\right.
        \end{equation*}
        In other words, there exists some universal $(\varepsilon,C)$ such that for any $y\in B_{\varepsilon h}(x)$, one has 
        \begin{equation}\label{eq. Green function near the pole}
            \left\{\begin{aligned}
            &\frac{C^{-1}}{2\pi}\ln{\frac{8h}{|x-y|}}\leq G(x,y)\leq\frac{C}{2\pi}\ln{\frac{8h}{|x-y|}},&\mbox{if }n=2,\\
            &C^{-1}\frac{|x-y|^{2-n}}{n(n-2)\omega_{n}}\leq G(x,y)\leq C\frac{|x-y|^{2-n}}{n(n-2)\omega_{n}},&\mbox{if }n\geq3.
            \end{aligned}\right.
        \end{equation}
        \item Step 3: Estimate of $G(x,y)$ in $\Omega\cap B_{2h}$. 
        
        As is shown in \eqref{eq. Green function near the pole}, for all $y\in 
        B_{\varepsilon h}(x)$, we obtain
        \begin{equation}\label{eq. Green function near the pole, repeat simpler}
           G(x,y)\sim \left\{\begin{aligned}
            &\frac{1}{2\pi}\ln{\frac{8h}{|x-y|}},&\mbox{if }n=2,\\
            &\frac{|x-y|^{2-n}}{n(n-2)\omega_{n}},&\mbox{if }n\geq3.
            \end{aligned}\right.
        \end{equation}
        It then remains to estimate $G(x,y)$ near the boundary. Note that $G(x,y)$ is a positive function vanishing at the boundary. Hence, we can apply Lemma~\ref{lem. classical BHP} to $G(x,y)$ near the boundary for all points $y$ staying away from the pole $x$. 
        
        To achieve this, we first recall that the homogeneous harmonic function $H(y)$ defined  in \eqref{eq. H Sigma, homogeneous harmonic in a cone} admits the following growth rate:
        \begin{equation*}
            H(y)=r^{\phi}E(\vec{\theta})\sim r^{\phi}\eta(y)=r^{\phi}\eta_{\Sigma}(\vec{\theta}),\,\, \mbox{for}\,\, y=r\vec{\theta}\in Cone_{\Sigma}.
        \end{equation*}
        Then, we apply the boundary Harnack principle (Lemma~\ref{lem. classical BHP}) to $G(x,y)$ and $H(y)$ in the region
        \begin{equation*}
            \mathcal{R}=\Omega\cap B_{3h}\cap \left(B_{\varepsilon h/2}(x)\right)^{c}.
        \end{equation*}
        In particular,  consider a reference point $y_{0}\in\partial B_{\varepsilon h}(x)$. Then for all $y\in\Omega\cap B_{2h}\cap \left(B_{\varepsilon h}(x)\right)^{c},$ we get 
        \begin{equation*}
            \frac{G(x,y)}{H(y)}\sim\frac{G(x,y_{0})}{H(y_{0})}\sim\left\{\begin{aligned}
                &\ln{\frac{8}{\varepsilon}}\cdot h^{-\phi},&\mbox{if }&n=2,\\
                &(\varepsilon h)^{2-n}\cdot h^{-\phi},&\mbox{if }&n\geq3,
            \end{aligned}\right.
        \end{equation*}
        where we have estimated $G(x,y_{0})$ using \eqref{eq. Green function near the pole, repeat simpler}. Consequently, for $y=r\vec{\theta}$ in $\Omega\cap B_{2h}\cap \left(B_{\varepsilon h}(x)\right)^{c}$, we have
        \begin{equation}\label{eq. Green function far from the pole, close to the corner}
           G(x,y)\sim \left\{\begin{aligned}
                &\ln{\frac{8}{\varepsilon}}\cdot(\frac{r}{h})^{\phi}\eta(y),&\mbox{if }&n=2,\\
                &(\varepsilon h)^{2-n}\cdot(\frac{r}{h})^{\phi}\eta(y),&\mbox{if }&n\geq3.
            \end{aligned}\right.
        \end{equation}
    \end{itemize}
    Finally, we combine \eqref{eq. Green function near the pole} and \eqref{eq. Green function far from the pole, close to the corner} and obtain the desired estimate.
\end{proof}

Next, we use the symmetry of the Green function, i.e. $G(x,y)=G(y,x)$ to study the behavior of $G(x,y)$ when $|y|\geq2h$.
\begin{lemma}[Green function for large $y$]\label{lem. green function for large y}
    For $x=h\vec{e_{1}}$ with $h\leq\frac{1}{8}$, and
    \begin{equation*}
        y=r\vec{\theta}\in\Omega\setminus\left(\Omega\cap B_{2h}\right),
    \end{equation*}
    we have
    \begin{equation*}
        G(x,y)\sim\frac{h^{\phi}}{r^{\phi+n-2}}\cdot(1-r)\cdot\eta(y).
    \end{equation*}
\end{lemma}
\begin{proof}
    We first consider the case that $r=|y|$ satisfies $r\leq\frac{1}{4}$. Let
    \begin{equation*}
        z=r\vec{e_{n}}.
    \end{equation*}
    Then, by the symmetry of the Green function, we apply Lemma~\ref{lem. green function for small y} with the pole being $z$, and obtain that
    \begin{equation*}
        G(x,z)=G(z,x)\sim\left\{\begin{aligned}
            &\ln{\frac{8r}{|x-z|}}\cdot(\frac{h}{r})^{\phi},&\mbox{if }&n=2\\
            &|x-z|^{2-n}(\frac{h}{r})^{\phi},&\mbox{if }&n\geq3
        \end{aligned}\right.\sim\frac{h^{\phi}}{r^{\phi+n-2}}.
    \end{equation*}
    Here, we have identified $|x-z|\sim r$ as $|z|=r\geq2h$. Then we apply Lemma~\ref{lem. classical BHP} in the annulus $Cone_{\Sigma}\cap(B_{\frac{11}{10}r}\setminus B_{\frac{9}{10}r})$, using $z$ as the the reference point yields 
    \begin{equation*}
        \frac{G(x,y)}{H(y)}\sim\frac{G(x,z)}{H(z)}\sim\frac{h^{\phi}}{r^{2\phi+n-2}},\quad\mbox{for }y\in Cone_{\Sigma}\cap(B_{\frac{11}{10}r}\setminus B_{\frac{9}{10}r}).
    \end{equation*}
    In other words, for $y\in Cone_{\Sigma}\cap(B_{\frac{11}{10}r}\setminus B_{\frac{9}{10}r})$,
    \begin{equation}\label{eq. Green function large radius, close to edge}
        G(x,y)\sim\frac{h^{\phi}}{r^{\phi+n-2}}\eta(y).
    \end{equation}

    On the other hand, if $r=|y|$ satisfies $r\geq\frac{1}{4}$, then we similarly obtain that
    \begin{equation*}
        G(x,\frac{1}{4}\vec{e_{n}})\sim h^{\phi}.
    \end{equation*}
    We then apply the boundary Harnack in the region $\Omega\setminus B_{1/6}$, but this time we use a different harmonic function to handle  the corner points in $\partial B_{1}\cap\partial Cone_{\Sigma}$. In fact, we set
    \begin{equation*}
        \mu(r)=r^{2-n-2\phi}-1,\quad r\in[\frac{1}{6},1],
    \end{equation*}
    and consider a harmonic function of the form:
    \begin{equation*}
        \widetilde{H}(y)=\mu(r)\cdot H(y),\quad y=r\vec{\theta}\in\Omega\setminus B_{1/6}.
    \end{equation*}
    In fact, recall that $H(y)$ is harmonic, and that $\partial_{r}H(y)=\frac{\phi}{r}H(y)$ from its $\phi$-homogeneity, we then have
    \begin{equation*}
        \Delta\widetilde{H}(y)=H(y)\Delta\mu(r)+2\frac{\phi}{r}H(y)\partial_{r}\mu(r)=H(y)\cdot\Big(\ddot{\mu}(r)+\frac{n-1+2\phi}{r}\dot{\mu}(r)\Big)=0,
    \end{equation*}
    showing that $\widetilde{H}(y)$ is indeed harmonic. Moreover, one sees that $\widetilde{H}(y)\sim(1-r)\cdot\eta(y)$ in $\Omega\setminus B_{1/6}$ and it vanishes on $\partial\Omega\setminus B_{1/6}$. By Lemma~\ref{lem. classical BHP}, we have 
    \begin{equation*}
        \frac{G(x,y)}{\widetilde{H}(y)}\sim\frac{G(x,\frac{1}{4}\vec{e_{n}})}{\widetilde{H}(\frac{1}{4}\vec{e_{n}})}\sim h^{\phi},\quad\mbox{for }y\in\Omega\setminus B_{1/6},
    \end{equation*}
    so
    \begin{equation}\label{eq. Green function, radius close to 1}
        G(x,y)\sim h^{\phi}\cdot(1-r)\cdot\eta(y).
    \end{equation}
    Combining \eqref{eq. Green function large radius, close to edge} with  \eqref{eq. Green function, radius close to 1} gives the desired estimate.
\end{proof}
\begin{remark}
    When $n=2$, one may replace the defining function $\eta(y)$ with $\cos{(\phi\theta)}$ in Lemma~\ref{lem. green function for small y} and Lemma~\ref{lem. green function for large y}, as mentioned in Example~\ref{ex. defining function n=2}.
\end{remark}

\section{Estimates of \texorpdfstring{$U(x)$}{Lg} in annuli}\label{sec. estimates in annuli}
\subsection{Reference points}
We consider a sequence of reference points
\begin{equation}\label{eq. reference points pk}
    p_{k}:=\frac{16^{1-k}}{2}\vec{e_{n}},\quad k\geq1.
\end{equation}
Let $U(x):\Omega\to\mathbb{R}_{+}$ be the solution to  \eqref{eq. special solution U(x)}. By means of  Lemma~\ref{lem. Harnack Comparable, general Lipschitz domain},  once the growth rate of $U(x)$ near the origin is  established, it follows that  any other solution of \eqref{eq. main} in $\Omega$ vanishing on $\partial Cone_{\Sigma}$ must have the same growth rate near the origin.

Next, we denote
\begin{equation}\label{eq. a_k definition}
    a_{k}=16^{k\phi}U(p_{k})=16^{\frac{2k}{1+\gamma}}U(p_{k}).
\end{equation}
We observe that $a_{k}$ is an increasing sequence, as stated below:
\begin{lemma}[Monotonicity of $a_{k}$]\label{lem. a_k is increasing}
    Let $U(x)$ and $a_{k}$ be defined as in \eqref{eq. special solution U(x)} and \eqref{eq. a_k definition}. Then $a_{1}\sim1$ and $a_{k+1}\geq a_{k}$ for every $k\geq1$.
\end{lemma}
\begin{proof}
    The first estimate $a_{1}\sim1$ can be obtained by constructing suitable barrier functions, and we omit the details. To see that $a_{k+1}\geq a_{k}$, we compare $U(x)$ with its invariant rescaling: 
    \begin{equation*}
        \widetilde{U}(x)=16^{\phi}\cdot U(\frac{x}{16})=16^{\frac{2}{1+\gamma}}\cdot U(\frac{x}{16}).
    \end{equation*}
    Note that $\widetilde{U}(x)$ also satisfies \eqref{eq. main} in $\Omega$, while its boundary data on $\partial\Omega$ is larger than that of $U(x)$. Therefore, we have $\widetilde{U}(x)\geq U(x)$. Evaluating at  $x=p_{k}$ gives the desired inequality  $a_{k+1}\geq a_{k}$.
\end{proof}

Now, for each $k\geq1$, we define the $k$-th annulus by
\begin{equation}\label{eq. annulus construction}
    \mathbb{A}_{k}:=\Omega\cap\Big(B_{r_{k}}\setminus B_{r_{k+1}}\Big),\quad r_{k}=16^{1-k}.
\end{equation}
We note the following simple geometric facts:
\begin{itemize}
    \item The point $p_{k}$ lies in the  interior  of the annulus $\mathbb{A}_{k}$ with $dist(p_{k},\partial\mathbb{A}_{k})\sim16^{-k}$;
    \item The outer radius of $\mathbb{A}_{k}$ is exactly twice the distance from the origin to $p_{k}$;
    \item For every $k\geq2$, we have $\frac{15}{16}\leq dist(\mathbb{A}_{k},\partial B_{1})\leq1$.
\end{itemize}

The key difficulty in this section is to prove the following asymptotic estimate in a $C^{1,1}$ domain. It improves the boundary estimate of Gui-Lin \cite{GuLi93}, in particular, when the $L^{\infty}$ norm of $u$ is huge, then our estimate is more accurate than \cite{GuLi93}.
\begin{theorem}[A localized growth rate estimate in a $C^{1,1}$ domain]\label{thm. by-product}
    Assume that $\Gamma=\{x_{n}=g(x')\}$, such that $\|g(x')\|_{C^{0,1}(B_{20}')}\leq L$, and that
    \begin{equation*}
        -\frac{1}{K}+\sqrt{\frac{1}{K^{2}}-|x'|^{2}}\leq g(x')\leq\frac{1}{K}-\sqrt{\frac{1}{K^{2}}-|x'|^{2}}\quad\mbox{in }B_{20}'
    \end{equation*}
    for some sufficiently small uniform $K$. Suppose that $u$ is a solution to \eqref{eq. main} with $\gamma>1$ in $\mathcal{GC}_{20}(0)$ and $u$ vanishes on $\Gamma$. Denote $\lambda=u(\vec{e_{n}})$, which clearly holds that $\lambda\gtrsim1$ (see Lemma~\ref{lem. review : interior lower bound}), then for $t\in[0,1]$, we have
    \begin{equation*}
        u(t\vec{e_{n}})\sim\left\{\begin{aligned}
            &t^{\phi},&\mbox{if }&t\leq\lambda^{\frac{1+\gamma}{1-\gamma}},\\
            &\lambda t,&\mbox{if }&t\geq\lambda^{\frac{1+\gamma}{1-\gamma}}.
        \end{aligned}\right.
    \end{equation*}
\end{theorem}
Before proving Theorem~\ref{thm. by-product}, let's first state its corollary below. 
Corollary~\ref{cor. estimate in a cone} provides an asymptotic estimate for 
$U^{-\gamma}$  in terms of the sequence $\{a_{k}\}_{k\geq1}$ (except in $\mathbb{A}_{1}$). This result is crucial for computing the right-hand side of the integral equation  \eqref{eq. key integral equation using Green function}.

\begin{cor}\label{cor. estimate in a cone}
    Let $Cone_{\Sigma}$ be a $C^{1,1}$ epigraphical cone with norms $(L,K)$, and let the annuli $\mathbb{A}_{k}$'s be given by \eqref{eq. annulus construction}. Assume that $u(x)$ satisfies \eqref{eq. main} with $\gamma>1$ in $\mathbb{A}_{k-1}\cup\mathbb{A}_{k}\cup\mathbb{A}_{k+1}$ for some $k\geq2$ and vanishes on $\partial Cone_{\Sigma}$. Let $a_{k}$ be defined as in \eqref{eq. a_k definition}. Then for each $x\in\mathbb{A}_{k}$, we have
    \begin{equation}\label{eq. asymptotic in the annulus}
        u(x)\sim\left\{\begin{aligned}
            &16^{-k\phi}\eta^{\phi},&\mbox{if }&\eta\leq a_{k}^{\frac{1+\gamma}{1-\gamma}},\\
            &16^{-k\phi}a_{k}\eta,&\mbox{if }&\eta\geq a_{k}^{\frac{1+\gamma}{1-\gamma}},
        \end{aligned}\right.
    \end{equation}
    where $\eta=\eta(x)$ is the defining function mentioned in Definition~\ref{def. defining function eta}.
\end{cor}
\begin{proof}
    First,  consider the invariant scaling
    \begin{equation*}
        u_{k}(y)=16^{-k\phi}U(16^{k}y),
    \end{equation*}
    under which the domain $\mathbb{A}_{k-1}\cup\mathbb{A}_{k}\cup\mathbb{A}_{k+1}$ transforms as follows:
    \begin{align*}
        &\mathbb{A}_{k-1}\leadsto Cone_{\Sigma}\cap(B_{256}\setminus B_{16}),\\
        &\mathbb{A}_{k}\leadsto Cone_{\Sigma}\cap(B_{16}\setminus B_{1}),\\
        &\mathbb{A}_{k+1}\leadsto Cone_{\Sigma}\cap(B_{1}\setminus B_{1/16}).
    \end{align*}
    Moreover, the reference point $p_{k}$ is mapped to $8\vec{e_{n}}$, and $u_{k}(8\vec{e_{n}})=a_{k}$. The desired estimate \eqref{eq. asymptotic in the annulus} is therefore equivalent to
    \begin{equation}\label{eq. asymptotic in the annulus simpler}
        u_{k}(x)\sim\left\{\begin{aligned}
            &\eta^{\phi},&\mbox{if }&\eta\leq a_{k}^{\frac{1+\gamma}{1-\gamma}},\\
            &a_{k}\eta,&\mbox{if }&\eta\geq a_{k}^{\frac{1+\gamma}{1-\gamma}}
        \end{aligned}\right.
    \end{equation}
    for  $x \in Cone_{\Sigma}\cap(B_{16}\setminus B_{1}).$
    
    By Lemma~\ref{lem. definiting function in a patch} (b) and Lemma~\ref{lem. review : interior lower bound} (2), we find 
    \begin{equation*}
        u_{k}(x)\sim a_{k}\mbox{ and }\eta(x)\sim1,\quad\mbox{for all }x\in Cone_{\Sigma}\cap(B_{100}\setminus B_{1/10})\cap\Big\{dist(\frac{x}{|x|},\partial\Sigma)\geq d_{0}\Big\}.
    \end{equation*}
    When $x\in (B_{16}\setminus B_{1})\cap\partial Cone_{\Sigma}$,  we can find a unique point $\widetilde{x}$ such that
    \begin{equation*}
        dist(\widetilde{x},\partial Cone_{\Sigma})=|\widetilde{x}-x|=d_{0}.
    \end{equation*}
    Then we have $u_{k}(\widetilde{x})\sim a_{k}$. We then perform the invariant scaling:
    \begin{equation*}
        v_{k,x}(y)=d_{0}^{-\phi}u_{k}(x+d_{0}\cdot\sigma y),
    \end{equation*}
   where $\sigma\in SO(n)$ is some rotation that maps the inner normal of $\partial Cone_{\Sigma}$ at $x$ to $\vec{e_{n}}$. In particular, we have
    \begin{equation*}
        v_{k,x}(\vec{e_{n}})=d_{0}^{-\phi}u_{k}(\widetilde{x})\sim a_{k}.
    \end{equation*}
    Note that as $d_{0}$ is sufficiently small, the boundary curvature of $\partial Cone_{\Sigma}$ becomes extremely small under the map $y\mapsto x+d_{0}\cdot\sigma y$. Thus, we can  apply Theorem~\ref{thm. by-product} to $v_{k,x}(y)$ to obtain the asymptotic estimate along  the line segment $[x,\widetilde{x}]$, which yields \eqref{eq. asymptotic in the annulus simpler}.
\end{proof}
\begin{remark}
    Corollary~\ref{cor. estimate in a cone} does not provide an estimate in the largest annulus $\mathbb{A}_{1}$, because $\mathbb{A}_{1}$ lacks a  neighboring annulus $\mathbb{A}_{0}$ inside $\Omega$. To estimate $U(x)$ in $\mathbb{A}_{1}$, we first note that $\|U\|_{L^{\infty}(\mathbb{A}_{0}\cup\mathbb{A}_{1})}\sim1$, and that every boundary point in $\partial\Omega\setminus B_{1/256}$ admits a sufficiently flat exterior cone. Therefore, we can apply Lemma~\ref{lem. subcritical boundary estimate} to obtain the following estimate:
    \begin{equation}\label{eq. estimate in A1}
        U(x)\sim \left(dist(x,\partial\Omega)\right)^{\frac{2}{1+\gamma}},\quad\mbox{for }x\in\mathbb{A}_{1}.
    \end{equation}
\end{remark}
In the case $n=2$, Corollary~\ref{cor. estimate in a cone} can also be can also be expressed  in the following polar coordination version, since one can choose $\eta(y)=\cos{(\phi\theta)}$ as mentioned in Example~\ref{ex. defining function n=2}.
\begin{exmp}\label{ex. growth rate in annuli n=2}
    When $n=2$, let $y=r\sin{\theta}\vec{e_{1}}+r\cos{\theta}\vec{e_{2}}$ (see Example~\ref{ex. 2d how large is the critical angle}) be a point in the annulus $\mathbb{A}_{k}$ for $k\geq2$, then in $\mathbb{A}_{k}$, one has 
    \begin{equation*}
        U(y)\sim\left\{\begin{aligned}
            &16^{-k\phi}\left(\cos{(\phi\theta)}\right)^{\phi},&\mbox{if }&\cos{(\phi\theta)}\leq a_{k}^{\frac{1+\gamma}{1-\gamma}},\\
            &16^{-k\phi}a_{k}\cos{(\phi\theta)},&\mbox{if }&\cos{(\phi\theta)}\geq a_{k}^{\frac{1+\gamma}{1-\gamma}}.
        \end{aligned}\right.
    \end{equation*}
\end{exmp}
\subsection{Proof of Theorem~\ref{thm. by-product}}
In this subsection, we prove Theorem~\ref{thm. by-product}. To this end, we first define and study a class of translation-invariant solutions. In fact, they are the only global solutions in $\mathbb{R}^{n}_{+}$ to \eqref{eq. main} vanishing on $\partial\mathbb{R}^{n}_{+}$, as previously  studied and classified in \cite{MoMuSc24a,MoMuSc24b}.
\begin{lemma}[Translation-invariant solutions]\label{lem. translation invariant solutions}
    Let $\gamma>1$ and let $V_{\lambda}:\mathbb{R}_{+}\to\mathbb{R}_{+}$ be an ODE solution to
    \begin{equation}\label{eq. 1d solution definition}
        \begin{cases}
            -V_{\lambda}''(t)=\left(V_{\lambda}(t)\right)^{-\gamma}&\mbox{for } \,\,t>0,\\
            V_{\lambda}'(t)\to\lambda&\mbox{when }\,\,t\to\infty,\\
            V_{\lambda}(0)=0.&
        \end{cases}
    \end{equation}
    Then there is a universal estimate for any $(\lambda,t)$, such that
    \begin{equation}\label{eq. 1d solution estimate}
        V_{\lambda}(t)\sim\left\{\begin{aligned}
            &t^{\phi},&\mbox{if }&t\leq\lambda^{\frac{1+\gamma}{1-\gamma}},\\
            &\lambda t,&\mbox{if }&t\geq\lambda^{\frac{1+\gamma}{1-\gamma}}.
        \end{aligned}\right.
    \end{equation}
    In particular, for $\lambda\gtrsim1$, we have $V_{\lambda}(1)\sim\lambda$. Moreover, the derivative of $V_{\lambda}$ satisfies
    \begin{equation}\label{eq. 1d solution estimate, derivative}
        V_{\lambda}'(t)\sim\left\{\begin{aligned}
            &t^{\phi-1},&\mbox{if }&t\leq\lambda^{\frac{1+\gamma}{1-\gamma}},\\
            &\lambda, &\mbox{if }&t\geq\lambda^{\frac{1+\gamma}{1-\gamma}}.
        \end{aligned}\right.
    \end{equation}
\end{lemma}
\begin{proof}
    We reduce the second-order equation \eqref{eq. 1d solution definition} into a first-order equation by multiplying $V_{\lambda}'(t)$ on both sides of \eqref{eq. 1d solution definition}. This gives 
    \begin{equation*}
        \frac{d}{dt}\Big(V_{\lambda}'(t)\Big)^{2}=2V_{\lambda}''(t)V_{\lambda}'(t)=-2V_{\lambda}(t)^{-\gamma}V_{\lambda}'(t)=\frac{2}{\gamma-1}\frac{d}{dt}\Big(V_{\lambda}(t)\Big)^{1-\gamma}.
    \end{equation*}
    Using  the asymptotic slope assumption $V_{\lambda}'(t)\to\lambda$, we obtain 
    \begin{equation}\label{eq. 1st integral of ODE}
        z'=\sqrt{\frac{2}{\gamma-1}z^{1-\gamma}+\lambda^{2}},
    \end{equation}
    where $z=V_{\lambda}(t)$. We then obtain the following integration identity:
    \begin{equation*}
        \int_{0}^{V_{\lambda}(t)}\frac{dz}{\sqrt{\frac{2}{\gamma-1}z^{1-\gamma}+\lambda^{2}}}=t.
    \end{equation*}
        It is straightforward to  verify that the integral on the left-hand side has the uniform estimate:
    \begin{equation*}
        \int_{0}^{s}\frac{dz}{\sqrt{\frac{2}{\gamma-1}z^{1-\gamma}+\lambda^{2}}}\sim\left\{\begin{aligned}
            &s^{1/\phi},&\mbox{if }&z\leq\lambda^{-\frac{2}{\gamma-1}},\\
            &s/\lambda,&\mbox{if }&z\geq\lambda^{-\frac{2}{\gamma-1}}.
        \end{aligned}\right.
    \end{equation*}
    Reversing such an estimate yields \eqref{eq. 1d solution estimate}. Finally, substituting \eqref{eq. 1d solution estimate} into \eqref{eq. 1st integral of ODE} gives the derivative estimate \eqref{eq. 1d solution estimate, derivative}.
\end{proof}
Next, we edit such a class of solutions, so that they provide the sharp estimate of radially symmetric solutions.
\begin{lemma}[Rotation-invariant solutions]\label{lem. rotational}
    Assume that $W_{\lambda,\pm R}:[0,3]\to\mathbb{R}_{+}$ is a rotation-invariant solution to \eqref{eq. main}, meaning that it satisfies the ODE problem ($\gamma>1$):
    \begin{equation*}
        \left\{\begin{aligned}
            &W_{\lambda,\pm R}''(t)+\frac{1}{t\pm R}W_{\lambda,\pm R}'(t)=-\left(W_{\lambda,\pm R}(t)\right)^{-\gamma}&\mbox{for }&t\in(0,3),\\
            &W_{\lambda,\pm R}(3)=\lambda,&&\\
            &V_{\lambda}(0)=0,&&
        \end{aligned}\right.
    \end{equation*}
    where $R$ is a sufficiently large uniform constant. Then for $\lambda\gtrsim1$, we have
    \begin{equation}\label{eq. rotational solution estimate}
        W_{\lambda,\pm R}(t)\sim\left\{\begin{aligned}
            &t^{\phi},&\mbox{if }&t\leq\lambda^{\frac{1+\gamma}{1-\gamma}},\\
            &\lambda t,&\mbox{if }&\lambda^{\frac{1+\gamma}{1-\gamma}}\leq t\leq1,\\
            &\lambda,&\mbox{if }&1\leq t\leq3.
        \end{aligned}\right.
    \end{equation}
\end{lemma}
\begin{proof}
    Rewrite the ODE in a more general form ($W_{\lambda}$ is a short-hand for $W_{\lambda,\pm R}$):
    \begin{equation}\label{eq. simplified rotational ODE}
        W_{\lambda}''(t)+c(t)W_{\lambda}'(t)=-\left(W_{\lambda}(t)\right)^{-\gamma},
    \end{equation}
where $|c(t)|\leq\varepsilon:=\frac{2}{R}$ is sufficiently small.
    
    To prove \eqref{eq. rotational solution estimate}, it suffices to construct suitable upper and lower barriers of \eqref{eq. simplified rotational ODE}.
    
    Let $M$ be sufficiently large, and construct an upper barrier:
    \begin{equation*}
        \overline{W_{\lambda}}(t)=M\cdot V_{\lambda}(t)-\frac{\lambda}{M^{2}}\Big[(t-\lambda^{\frac{1+\gamma}{1-\gamma}})_{+}\Big]^{2},
    \end{equation*}
    and a lower barrier
    \begin{equation*}
        \underline{W_{\lambda}}(t)=\frac{1}{M}\cdot V_{\lambda}(t)+\frac{\lambda}{M^{2}}\Big[(t-\lambda^{\frac{1+\gamma}{1-\gamma}})_{+}\Big]^{2}.
    \end{equation*}
    We require $\varepsilon=\frac{2}{R}$ to be sufficiently small (i.e., $R$ sufficiently large), such that $\varepsilon\leq M^{-4}$.
    
    We now check that they do are barriers. First, when $M$ is large, it holds that 
    \begin{equation*}
        \frac{\lambda}{M^{2}}\Big[(t-\lambda^{\frac{1+\gamma}{1-\gamma}})_{+}\Big]^{2}\leq\frac{3\lambda}{M^{2}}t\leq\frac{1}{M}\cdot V_{\lambda}(t)\quad\mbox{for all }0\leq t\leq3,
    \end{equation*}
   which implies that
    \begin{equation*}
        \overline{W_{\lambda}}(t)\geq\frac{M}{2}\cdot V_{\lambda}(t)\quad\mbox{and}\quad\underline{W_{\lambda}}(t)\leq\frac{2}{M}\cdot V_{\lambda}(t)\quad\mbox{for all }0\leq t\leq3.
    \end{equation*}
    In particular, we see
    \begin{equation*}
        \underline{W_{\lambda}}(3)<\lambda<\overline{W_{\lambda}}(3).
    \end{equation*}
    Second, let's check that they satisfy the ODE inequality, provided that $M$ is large. For $\overline{W_{\lambda}}$, since $\overline{W_{\lambda}}(t)\geq\frac{M}{2}\cdot V_{\lambda}(t)$, we derive the following estimate (in the viscosity sense):
    \begin{align*}
    &\overline{W_{\lambda}}''(t)+c(t)\overline{W_{\lambda}}'(t)+\left(\overline{W_{\lambda}}(t)\right)^{-\gamma}\\
        \leq&M\cdot V_{\lambda}''(t)-\frac{2\lambda}{M^{2}}\cdot\chi_{(\lambda^{\frac{1+\gamma}{1-\gamma}},\infty)}+M\varepsilon\cdot V_{\lambda}'(t)+\frac{2\varepsilon\lambda}{M^{2}}(t-\lambda^{\frac{1+\gamma}{1-\gamma}})_{+}+\Big(\frac{M}{2}\cdot V_{\lambda}(t)\Big)^{-\gamma}\\
        \leq&-\Big(M-\frac{2^{\gamma}}{M^{\gamma}}\Big)\left(V_{\lambda}(t)\right)^{-\gamma}-\frac{2\lambda}{M^{2}}\cdot\chi_{(\lambda^{\frac{1+\gamma}{1-\gamma}},\infty)}+M\varepsilon\cdot V_{\lambda}'(t)+\frac{6\varepsilon\lambda}{M^{2}}\\
        \leq&-\left(V_{\lambda}(t)\right)^{-\gamma}-\frac{2\lambda}{M^{2}}\cdot\chi_{(\lambda^{\frac{1+\gamma}{1-\gamma}},\infty)}+\frac{2}{M^{3}}\cdot V_{\lambda}'(t)\leq 0
    \end{align*}
    for $0\leq t\leq 3,$
    where the last step follows from Lemma~\ref{lem. translation invariant solutions}. Similarly,  since $\underline{W_{\lambda}}(t)\leq\frac{2}{M}\cdot V_{\lambda}(t)$, then for $0\leq t \leq 3$, we obtain
    \begin{align*}   &\underline{W_{\lambda}}''(t)+c(t)\underline{W_{\lambda}}'(t)+\left(\underline{W_{\lambda}}(t)\right)^{-\gamma}\\
        \geq&\frac{1}{M}\cdot V_{\lambda}''(t)+\frac{2\lambda}{M^{2}}\cdot\chi_{(\lambda^{\frac{1+\gamma}{1-\gamma}},\infty)}-\frac{\varepsilon}{M}\cdot V_{\lambda}'(t)-\frac{2\varepsilon\lambda}{M^{2}}(t-\lambda^{\frac{1+\gamma}{1-\gamma}})_{+}+\Big(\frac{2}{M}\cdot V_{\lambda}(t)\Big)^{-\gamma}\\
        \geq&\Big(\frac{M^{\gamma}}{2^{\gamma}}-\frac{1}{M}\Big)\left(V_{\lambda}(t)\right)^{-\gamma}+\frac{2\lambda}{M^{2}}\cdot\chi_{(\lambda^{\frac{1+\gamma}{1-\gamma}},\infty)}-\frac{\varepsilon}{M}\cdot V_{\lambda}'(t)-\frac{6\varepsilon\lambda}{M^{2}}\\
        \geq&\left(V_{\lambda}(t)\right)^{-\gamma}+\frac{2\lambda}{M^{2}}\cdot\chi_{(\lambda^{\frac{1+\gamma}{1-\gamma}},\infty)}-\frac{2}{M^{5}}\cdot V_{\lambda}'(t)\geq0.
    \end{align*}
 Therefore, we infer from the maximum principle that
    \begin{equation*}
        \underline{W_{\lambda}}(t)\leq W_{\lambda,\pm R}(t)\leq\overline{W_{\lambda}}(t)\quad\mbox{for }0\leq t\leq3.
    \end{equation*}
    Note that the barriers $\overline{W_{\lambda}}(t)$ and $\underline{W_{\lambda}}(t)$ both satisfy the estimate \eqref{eq. rotational solution estimate}, then so does $W_{\lambda,\pm R}(t)$.
\end{proof}

Now, we are able to prove Theorem~\ref{thm. by-product}. The proof is based on constructing upper and lower barriers.
\begin{proof}[Proof of Theorem~\ref{thm. by-product}]
    From Lemma~\ref{lem. u is small near the boundary}, we deduce that (by fixing some small universal $\delta$):
    \begin{equation*}
        \max_{\mathcal{GC}_{6}(0)}u\sim\min_{\mathcal{SC}_{6,\delta}(0)}u\sim\lambda.
    \end{equation*}
    If $K$ is sufficiently small, then at the origin, the curve $\{x_{n}=g(x')\}$ admits both  an interior and an exterior tangential ball with a large radius $R$ (the same as the one in Lemma~\ref{lem. rotational}). Moreover, for every $x'$ satisfying $1\leq|x'|\leq6$, we have a geometric observation:
    \begin{equation}\label{eq. thm 3.1 geometric observation}
        R-\sqrt{R^{2}-|x'|^{2}}\geq\frac{1}{K}-\sqrt{\frac{1}{K^{2}}-|x'|^{2}}+6\delta\geq g(x')+6\delta.
    \end{equation}

    Now, we define two domains (see Figure 2 and Figure 3 below):
    \begin{align*}
        \mathcal{D}_{1}=\{x=(x',x_{n}):|x'|\leq5,\ 0\leq x_{n}+R-\sqrt{R^{2}-|x'|^{2}}\leq5\},\\
        \mathcal{D}_{2}=\{x=(x',x_{n}):|x'|\leq5,\ 0\leq x_{n}-R+\sqrt{R^{2}-|x'|^{2}}\leq5\}.
    \end{align*}
    Note that in Figure 3, there is a $6\delta$-gap at $|x'|=5$ between $\mathcal{D}_{2}$ and $\Gamma$, and this is the geometric interpretation of \eqref{eq. thm 3.1 geometric observation}.

 \begin{center}
 \begin{tikzpicture}[node distance = 0.1cm]
 \draw[red,line width=1.0] (2.828,-1.172) arc (45:135:4);
\draw[red,line width=1.0] (2.828,2.328) arc (45:135:4)
node at (2, 2.1) [ font=\fontsize{15}{15}\selectfont] {$ \mathcal{D}_1$};
 \path (0,0) [thick,fill=black]  circle(1.3 pt) node at (0, -0.4) [ font=\fontsize{15}{15}\selectfont] {$0$};
\draw [thick][red] (-2.828,-1.172) -- (-2.828,2.328);
\draw [thick] [red] (2.828,-1.172) -- (2.828,2.328);
\path (0,1.5) [thick,fill=black]  circle(1.3 pt) node at (0.4, 1.5) [ font=\fontsize{15}{15}\selectfont] {$\vec{e_1}$};
\draw[black,domain=-3.5:3.5,smooth] plot(\x, {0.02*\x*\x*\x}) 
node at (3.3, 1) [ font=\fontsize{15}{15}\selectfont] {$\Gamma$};
\node [below=1cm, align=flush center,text width=8cm] at (0,-1.5)
        {$Figure$ 2 \fontsize{15}{15}\selectfont };
 \draw[red,line width=1.0] (5.828,3.172) arc (225:315:4);
  \draw[red,line width=1.0] (5.828,-0.328) arc (225:315:4)
  node at (10.5, 1.7) [ font=\fontsize{15}{15}\selectfont] {$ \mathcal{D}_2$};
  \draw [thick][red] (5.828,3.172)--(5.828,-0.328);
    \draw [thick][red] (11.46,3.172)--(11.46,-0.328);
     \path (8.644,-1.5) [thick,fill=black]  circle(1.3 pt) node at (8.644, -1.75) [ font=\fontsize{15}{15}\selectfont] {$0$};
     \path (8.644,-0.2) [thick,fill=black]  circle(1.3 pt) node at (9, -0.2) [ font=\fontsize{15}{15}\selectfont] {$\vec{e_1}$};
     \draw[black,domain=5.5:12,smooth] plot(\x, {0.02*(\x-8)*(\x-8)*(\x-8)-1.52})
     node at (12.2, -0.7) [ font=\fontsize{15}{15}\selectfont] {$\Gamma$};
     \draw [thick][blue] (8.044,-0.8)--(8.044,0.4);
     \draw [thick][blue] (9.244,-0.8)--(9.244,0.4);
      \draw [thick][blue] (8.044,-0.8)--(9.244,-0.8);
       \draw [thick][blue] (8.044,0.4)--(9.244,0.4);
       \draw [->, blue, thick] 
        (8.4,-0.8) to [out=0, in=200] (10.4,-1.6)
        node at (11.4, -1.8) [font=\fontsize{12}{12}\selectfont] 
        {$B_1' \times [\frac{1}{2}, \frac{3}{2}]$};
         \draw[<->,line width=1.2] [blue](11.46,-0.7) -- (11.46,-0.328) node at (11.78, 0.47) [ font=\fontsize{9}{9}\selectfont] [blue]{$>6\delta$};
           \draw[green, thick] (11.46,-0.48) -- (11.67,-0.48);
        \draw[->, green, thick] (11.68,-0.514) -- (11.68,0.3);
        \node [below=1cm, align=flush center,text width=8cm] at (8.644,-1.5)
        {$Figure$ 3 \fontsize{15}{15}\selectfont };
\end{tikzpicture}
 \end{center}

    Let's consider two solutions $\overline{u}$ and $\underline{u}$ as the upper and lower barrier of $u$, so that they are solutions to:
    \begin{equation*}
        \begin{cases}
            -\Delta\overline{u}(x)=\left(\overline{u}(x)\right)^{-\gamma}&\mbox{in }\,\,\mathcal{D}_{1},\\
            \overline{u}(y)=0&\mbox{on }\,\,\partial\mathcal{D}_{1}\cap\{x_{n}\leq g(x')\},\\
            \overline{u}(y)=u(y)&\mbox{on }\,\,\partial\mathcal{D}_{1}\cap\{x_{n}\geq g(x')\},
        \end{cases}
    \end{equation*}
    and
    \begin{equation*}
        \begin{cases}
            -\Delta\underline{u}(x)=\left(\underline{u}(x)\right)^{-\gamma}&\mbox{in }\,\,\mathcal{D}_{2},\\
            \underline{u}(y)=\Big(\frac{y_{n}-R+\sqrt{R^{2}-|y'|^{2}}}{10}\Big)\cdot\min_{\mathcal{SC}_{6,\delta}(0)}u&\mbox{on }\,\,\partial\mathcal{D}_{2}.
        \end{cases}
    \end{equation*}
    In fact, we see that the boundary value of $\overline{u}(x)$ equals that of $u(x)$ (which is extended trivially below $\Gamma=\{x_{n}=g(x')\}$), while $u(x)$ is a sub-solution to \eqref{eq. main} in $\mathcal{D}_{1}$ after the trivial extension. Therefore,
    \begin{equation*}
        u(x)\leq\overline{u}(x)\quad\mbox{in }\mathcal{D}_{1}\cap\mathcal{GC}_{6}(0).
    \end{equation*}
    We apply a similar argument for $\underline{u}$, but in this case, there is no need to  extend $u$ trivially because $u$ satisfies \eqref{eq. main} everywhere in $\mathcal{D}_{2}\subseteq\mathcal{GC}_{6}(0)$. We then partition $\partial\mathcal{D}_{2}$ into two distinct  parts, namely
    \begin{equation*}
        (\partial\mathcal{D}_{2}\cap\{x_{n}=R-\sqrt{R^{2}-|x'|^{2}}\})\cup(\partial\mathcal{D}_{2}\cap\{x_{n}>R-\sqrt{R^{2}-|x'|^{2}}\}).
    \end{equation*}
    In the first part, we have $\underline{u}(y)=0\leq u(y)$. In the second part, we have $\displaystyle \underline{u}(y)=\min_{\mathcal{SC}_{6,\delta}(0)}u\leq u(y)$ because of the geometric observation \eqref{eq. thm 3.1 geometric observation}. Therefore,
    \begin{equation*}
        u(x)\geq\underline{u}(x)\quad\mbox{in }\mathcal{D}_{2}.
    \end{equation*}
    
    To conclude, in the common domain $\mathcal{D}_{1}\cap\mathcal{D}_{2}\cap\mathcal{GC}_{6}(0)$, in particular along the line segment $[0,4\vec{e_{n}}]$, we have that
    \begin{equation}\label{eq. thm 1.3: upper and lower bound on a line}
        \underline{u}(t\vec{e_{n}})\leq u(t\vec{e_{n}})\leq\overline{u}(t\vec{e_{n}}),\quad\mbox{for all }0\leq t\leq4.
    \end{equation}
    It then suffices to bound $\underline{u}(t\vec{e_{n}})$ from below and $\overline{u}(t\vec{e_{n}})$ from above. In fact, the first key observation is that:
    \begin{equation}\label{eq. thm 3.1 claim}
        \underline{u}(\vec{e_{n}})\sim\overline{u}(\vec{e_{n}})\sim\lambda.
    \end{equation}
    Granted \eqref{eq. thm 3.1 claim}, we can immediately complete the proof of Theorem~\ref{thm. by-product}. In fact, we apply Corollary~\ref{cor. nonlinear boundary harnack accurate} to the pair $(\overline{u},W_{\lambda,+R})$ and the pair $(\underline{u},W_{\lambda,-R})$, then \eqref{eq. thm 3.1 claim} implies that
    \begin{equation*}
        \overline{u}(t\vec{e_{n}})\sim W_{\lambda,+R}(t)\quad\mbox{and}\quad\underline{u}(t\vec{e_{n}})\sim W_{\lambda,-R}(t)\quad\mbox{for }0\leq t\leq1.
    \end{equation*}
    Then an application of the growth rate estimate Lemma~\ref{lem. rotational} yields the conclusion of Theorem~\ref{thm. by-product}.

    Finally, we prove \eqref{eq. thm 3.1 claim}. First, from \eqref{eq. thm 1.3: upper and lower bound on a line}, we already have
    \begin{equation*}
        \underline{u}(\vec{e_{n}})\lesssim\lambda\lesssim\overline{u}(\vec{e_{n}}).
    \end{equation*}
    Next, using the boundary data of $\underline{u}$ and the fact that $\underline{u}$ is super-harmonic, we can bound $u(x)$ from below by a linear harmonic function, i.e.
    \begin{equation*}
        \underline{u}(x)\geq\frac{x_{n}-\frac{1}{2}}{200}\min_{\mathcal{SC}_{6,\delta}(0)}u\quad\mbox{in }B_{1}'\times[\frac{1}{2},\frac{3}{2}].
    \end{equation*}
    Then we have
    \begin{equation*}
        \underline{u}(\vec{e_{n}})\geq\frac{1}{400}\min_{\mathcal{SC}_{6,\delta}(0)}u\sim\lambda.
    \end{equation*}
    Third, observe that the boundary data of $\overline{u}$ is less than $\displaystyle \max_{\mathcal{GC}_{6}(0)}u$, and
    \begin{equation*}
        \max_{\mathcal{GC}_{6}(0)}u\sim\lambda\gtrsim1.
    \end{equation*}
    We then compare $\overline{u}$ with a paraboloid-shaped upper solution (with $M$ being a sufficiently large uniform constant):
    \begin{equation*}
        P(x)=M\lambda+(100n-|x|^{2}),\quad x\in B_{10\sqrt{n}}.
    \end{equation*}
    It then implies that
    \begin{equation*}
        \overline{u}(\vec{e_{n}})\leq P(\vec{e_{n}})\sim\lambda.
    \end{equation*}
    Combining these estimates  proves  \eqref{eq. thm 3.1 claim}. This completes  the proof of Theorem~\ref{thm. by-product}.
\end{proof}

\section{Derivation of a discrete integral equation}\label{sec. integral equation}
In this section, we analyze the integral equation
\begin{equation*}
    U(x)=\int_{\Omega}G(x,y)U^{-\gamma}(y) dy.
\end{equation*}
In particular, we focus on choosing $x=p_{k}$ as mentioned in \eqref{eq. reference points pk}. This leads us to derive  the following discrete integral equation for the sequence $a_{k}$ defined in \eqref{eq. a_k definition}, which will be crucial  for proving Theorem~\ref{thm. asymptotic estimate}.
\begin{lemma}[A discrete integral equation]\label{lem. discrete integral equation}
    Under the assumptions of Theorem~\ref{thm. asymptotic estimate}, let $\{a_{k}\}_{k\geq1}$ be the sequence defined in \eqref{eq. a_k definition}. Then for every $k\geq 1,$ $a_{k}$ satisfies the following discrete integral equation:
    \begin{equation*}
        a_{k}\sim\left\{\begin{aligned}
            &\sum_{j=1}^{k}a_{j}^{-\gamma},&\mbox{if }&1<\gamma<2,\\
            &\sum_{j=1}^{k}a_{j}^{-2}\ln{a_{j}},&\mbox{if }&\gamma=2,\\
            &\sum_{j=1}^{k}a_{j}^{\frac{2}{1-\gamma}},&\mbox{if }&\gamma>2.
        \end{aligned}\right.
    \end{equation*}
\end{lemma}

Since the Green function exhibits different behaviors between the case $n=2$ and the case $n\geq3$, we need to prove Lemma~\ref{lem. discrete integral equation} separately in two subsections.
\subsection{Case \texorpdfstring{$n=2$}{Lg}}

We first prove Lemma~\ref{lem. discrete integral equation} in dimension $2$. By the representation of the Green function, we have 
\begin{equation*}
U(p_k)=\int_\Omega G(p_k, y) U^{-\gamma} (y)dy. 
\end{equation*}

Now we divide the domain $\Omega$ into four regions (see Figure 4 below):
$$
\Omega_1:= \mathbb{A}_{k},
$$
$$
\Omega_2:=\Omega \cap B_{r_{k+1}}, 
$$
$$
\Omega_3:= \mathbb{A}_2\cup \mathbb{A}_3\cup \dots \cup\mathbb{A}_{k-1} , 
$$
and 
$$
\Omega_4=\mathbb{A}_1,
$$
with 
$$
 \mathbb{A}_{k}=\Omega\cap\Big(B_{r_{k}}\setminus B_{r_{k+1}}\Big)\quad \mbox{and}\quad r_{k}=16^{1-k}.
$$

\begin{center}
\begin{tikzpicture}[node distance = 0.1cm]
\draw[black,line width=1.0] (4.242,-4.242) arc (-45:225:6)
  node at (-3.6, 3.8) [red, font=\fontsize{17}{17}\selectfont] {$\Omega_4$};
 \draw [thick] (-4.242,-4.242) -- (0,0);
  \draw [thick] (4.242,-4.242) -- (0,0);
  \draw [thick]  [dashed] [->,thick](-8,0)--(8,0) node [anchor=north west] {$x'$};
    \draw [thick]  [dashed] [->,thick](0,-5.5)--(0,8) node [anchor=north west] {$x_n$};
    \draw[blue,line width=1.0] (3.182,-3.182) arc (-45:225:4.5)
     node at (-2.4, 2.8) [red, font=\fontsize{17}{17}\selectfont] {$\Omega_3$};
    \draw[green,line width=1.0] (2.121,-2.121) arc (-45:225:3)
     node at (-1.2, 2.2) [red, font=\fontsize{17}{17}\selectfont] {$\Omega_1$};
      \draw[blue,line width=1.0] (1.414,-1.414) arc (-45:225:2)    
      node at (0.4, 1) [red, font=\fontsize{17}{17}\selectfont] {$\Omega_2$};
       \path (0,2.5) [thick,fill=black]  circle(1.3 pt) node at (0.2, 2.5) [ font=\fontsize{9}{9}\selectfont] {$p_k$};
\draw [thick] [dashed] [red, line width=1.2](0, 2.5) circle (0.4)
node at (1, 2.35) [red, font=\fontsize{9}{9}\selectfont] {$B_{\varepsilon_k}(p_k)$};
        \draw[red,line width=1.2] [dashed](-4,-2) arc (280:315:7.5);
        \draw[red,line width=1.2] [dashed](4,-2) arc (260:225:7.5);
         \node [below=1cm, align=flush center,text width=8cm] at (0,-5)
        {$Figure$ 4 \fontsize{15}{15}\selectfont };
        \def\mypath{(0,0) -- (-3.182, -3.182) arc(225:210:4.5) -- (-4,-2) arc (280:315:7.5)}
       \pattern[pattern color=gray!80, pattern=north west lines]  \mypath;
          \def\mypath{(4,-2) arc (260:225:7.5) -- (0, 0)-- (3.182,-3.182)arc(315:330:4.5) }
       \pattern[pattern color=gray!80, pattern=north west lines]  \mypath;
\end{tikzpicture}
\end{center}

Then 
\begin{equation}\label{eq. J1-J4 decompose}
\begin{aligned}
U(p_k)=&\int_{\Omega_1} G(p_k, y) U^{-\gamma} (y)dy+\int_{\Omega_2} G(p_k, y) U^{-\gamma} (y)dy\\
&+\int_{\Omega_3} G(p_k, y) U^{-\gamma} (y)dy+\int_{\Omega_4} G(p_k, y) U^{-\gamma} (y)dy\\
=:& J_1+J_2+J_3+J_4.
\end{aligned}
\end{equation}

Let's also denote 
$$
I_j:=\int_{\mathbb{A}_j} G(p_k, y) U^{-\gamma} (y)dy,\quad j=1, 2, \cdots, k, \cdots.
$$

 \begin{itemize}
        \item Step 1: Estimate the term $J_1$ in \eqref{eq. J1-J4 decompose}.
 
        Noting that $p_k \in \Omega_1$ is a singular point,  we compute the integral in $\mathbb{A}_k \cup B_{\varepsilon_k}(p_k)$ and 
     $B_{\varepsilon_k}(p_k)$ separately, where $\varepsilon>0$ is a uniform and sufficiently small  constant, and  $\varepsilon_k= \varepsilon \cdot 16^{-k}$. 

     Let $y=(r\sin \theta, r\cos \theta)$, as described  in Example~\ref{ex. 2d how large is the critical angle}, and set
     $$
     D_{1, k}= (\mathbb{A}_k \backslash B_{\varepsilon_k}(p_k)) \cap \left\{\theta \in [-\frac{1+\gamma}{4}\pi, \frac{1+\gamma}{4}\pi] \mid cos(\phi \theta)\leq a_k^{\frac{1+\gamma}{1-\gamma}}\right\}
     $$
     and 
     $$
     D_{2, k}= (\mathbb{A}_k \backslash B_{\varepsilon_k}(p_k)) \cap \left\{\theta \in [-\frac{1+\gamma}{4}\pi, \frac{1+\gamma}{4}\pi] \mid cos(\phi \theta)> a_k^{\frac{1+\gamma}{1-\gamma}}\right\}.
     $$

In Figure 4,  $D_{1,k}$ denotes the gray shaded region within $\Omega_1$,  while  $D_{2,k}$  represents the portion of  $\Omega_1$ that excludes both the shadowed region and the ball $B_{\varepsilon_k}(p_k)$.

\begin{eqnarray}\label{eq. J1 decompose n=2}
\begin{aligned}
    J_1=& \int_{\mathbb{A}_k} G(p_k, y) U^{-\gamma} (y)dy \\
    =& \int_{D_{1, k}} G(p_k, y) U^{-\gamma} (y)dy +\int_{D_{2, k}} G(p_k, y) U^{-\gamma} (y)dy +\int_{B_{\varepsilon_k}(p_k)} G(p_k, y) U^{-\gamma} (y)dy \\
   =:& J_{11}+J_{12}+J_{13}.
\end{aligned}
\end{eqnarray}

Since
\begin{equation*}
    2|p_k|=16^{1-k}=r_k>r_{k+1}\mbox{ and }\ln \frac{8|p_k|}{|p_k-y|}\sim 1\mbox{ in }D_{1, k}\cup D_{2, k},
\end{equation*}
and that $\frac{r}{|p_k|}\sim 1$ in $\mathbb{A}_k$, then by Lemma~\ref{lem. green function for small y} and Example~\ref{ex. growth rate in annuli n=2}, we obtain  
\begin{eqnarray}\label{eq. J11 n=2}
\begin{aligned}
    J_{11}=& \int_{D_{1, k}} G(p_k, y) U^{-\gamma} (y)dy \\
    \sim &  \int_{D_{1, k}} \ln \frac{8|p_k|}{|p_k-y|} \left(\frac{r}{|p_k|}\right)^\phi \cos(\phi \theta) \left(16^{-k\phi} \left(\cos(\phi\theta)\right)^\phi \right)^{-\gamma} dy\\
    \sim & 16^{k\phi \gamma} \int_{D_{1, k}} \left(\cos(\phi \theta)\right)^{1-\phi\gamma} dy\\
    \leq & 16^{k\phi \gamma}\cdot 16^{-2k} \int_{0}^{a_k^{\frac{1+\gamma}{1-\gamma}}} t^{1-\phi\gamma} dt \quad(\mbox{here }t:=\cos(\phi \theta))\\
    \sim & 16^{-k\phi} a_k^{\frac{2}{1-\gamma}},
\end{aligned}
\end{eqnarray}
where we have use the fact that $\phi=\frac{2}{1+\gamma}$ and $1-\phi \gamma=\frac{1-\gamma}{1+\gamma}\in (-1, 0)$. 

For the term $J_{12}$,  by Lemma~\ref{lem. green function for small y} and Example~\ref{ex. growth rate in annuli n=2}, we have 
\begin{eqnarray}\label{eq. J12 n=2}
\begin{aligned}
    J_{12}=& \int_{D_{2, k}} G(p_k, y) U^{-\gamma} (y)dy \\
    \sim &  \int_{D_{2, k}} \ln \frac{8|p_k|}{|p_k-y|} \left(\frac{r}{|p_k|}\right)^\phi \cos(\phi \theta) \left(16^{-k\phi} a_k \cos(\phi\theta) \right)^{-\gamma} dy\\
    \sim & 16^{k\phi \gamma} a_k^{-\gamma}\int_{D_{2, k}} \left(\cos(\phi \theta)\right)^{1-\gamma} dy\\
    \leq & 16^{k\phi \gamma} a_k^{-\gamma}16^{-2k} \int_{a_k^{\frac{1+\gamma}{1-\gamma}}}^{1} t^{1-\gamma} dt \\
    \sim & 
    \begin{cases}
     16^{-k\phi} a_k^{-\gamma} , \quad &\mbox{if} \,\, 1<\gamma<2,\\
    16^{-k\phi} a_k^{-\gamma} \ln a_k, \quad &\mbox{if} \,\, \gamma =2,\\
     16^{-k\phi} a_k^{\frac{2}{1-\gamma}}, \quad &\mbox{if} \,\, \gamma >2.
    \end{cases}
\end{aligned}
\end{eqnarray}
Notice that in the last step, we have used the fact that
\begin{equation*}
    \int_{\lambda}^{1}t^{1-\gamma}dt\sim\begin{cases}
        1, \quad &\mbox{if} \,\, 1<\gamma<2,\\
    -\ln\lambda, \quad &\mbox{if} \,\, \gamma =2,\\
     \lambda^{2-\gamma}, \quad &\mbox{if} \,\, \gamma >2.
    \end{cases}
\end{equation*}

For the last term $J_{13}$  in \eqref{eq. J1 decompose n=2}, since  $cos(\phi \theta)\sim 1$ in $B_{\varepsilon_k}(p_k)$, one can  derive from Lemma~\ref{lem. green function for small y} and Example~\ref{ex. growth rate in annuli n=2} that 
\begin{eqnarray}\label{eq. J13 n=2}
\begin{aligned}
    J_{13}=& \int_{B_{\varepsilon_k}(p_k)} G(p_k, y) U^{-\gamma} (y)dy \\
    =& \int_{B_{\varepsilon_k}(p_k)} \ln \frac{8|p_k|}{|p_k-y|} \left(\frac{r}{|p_k|}\right)^\phi \cos(\phi \theta) \left(16^{-k\phi} a_k \cos(\phi\theta) \right)^{-\gamma} dy\\
    \sim &  16^{k\phi \gamma} a_k^{-\gamma}\int_{B_{\varepsilon_k}(p_k)} \ln \frac{8|p_k|}{|p_k-y|} dy\\
     \sim &  16^{k\phi \gamma} a_k^{-\gamma} \int_0^{2\pi}\int_0^{\varepsilon_k} \ln( \frac{8|p_k|}{\rho}) \rho d\rho d\theta\\
     \sim & 16^{k\phi \gamma} a_k^{-\gamma} |p_k|^2 \int_0^{\frac{\varepsilon}{64}} (-\ln \rho_1) \rho_1 d\rho_1\quad(\mbox{here }\rho_{1}:=\frac{\rho}{8|p_{k}|})\\
     \sim &16^{k\phi \gamma-2k} a_k^{-\gamma} \\
     \sim & 16^{-k\phi}   a_k^{-\gamma}, 
      \end{aligned}
\end{eqnarray}
where we have used the fact that 
$\frac{\varepsilon_k}{8|p_k|}=\frac{\varepsilon}{64}$.

Combining \eqref{eq. J1 decompose n=2}, \eqref{eq. J11 n=2}, \eqref{eq. J12 n=2} and \eqref{eq. J13 n=2} yields 
\begin{equation}\label{eq. J1 n=2}
    J_1\sim 
    \begin{cases}
     16^{-k\phi} a_k^{-\gamma}, \quad &\mbox{if}\,\, 1<\gamma <2,\\
    16^{-k\phi} a_k^{-2} \ln a_k, \quad &\mbox{if} \,\, \gamma =2,\\
     16^{-k\phi} a_k^{\frac{2}{1-\gamma}}, \quad &\mbox{if}\,\, \gamma >2.
    \end{cases}
\end{equation}

   \item Step 2: Estimate the term $J_2$ in \eqref{eq. J1-J4 decompose}.

Denote a shadowed region
 $$
     H_{1, j}= \mathbb{A}_j \cap \left\{\theta \in [-\frac{1+\gamma}{4}\pi, \frac{1+\gamma}{4}\pi] \mid cos(\phi \theta)\leq a_k^{\frac{1+\gamma}{1-\gamma}}\right\}
     $$
     and an unshadowed region
     $$
     H_{2, j}= \mathbb{A}_j  \cap \left\{\theta \in [-\frac{1+\gamma}{4}\pi, \frac{1+\gamma}{4}\pi] \mid cos(\phi \theta)> a_k^{\frac{1+\gamma}{1-\gamma}}\right\}.
     $$
Since for $j\geq k+1$:
\begin{equation*}
    \ln \frac{8|p_k|}{|p_k-y|}\sim 1\mbox{ and }\frac{r}{|p_k|}\sim 16^{k-j}\mbox{ in }H_{1, j}\cup H_{2, j},
\end{equation*}
then by Lemma~\ref{lem. green function for small y} and Example~\ref{ex. growth rate in annuli n=2} we have 
\begin{eqnarray*}
\begin{aligned}
I_j=&\int_{\mathbb{A}_j} G(p_k, y) U^{-\gamma} (y) dy\\
=& \int_{H_{1, j}} G(p_k, y) U^{-\gamma} (y) dy + \int_{H_{2, j}} G(p_k, y) U^{-\gamma} (y) dy\\
\sim & \int_{H_{1, j}}  \ln \frac{8|p_k|}{|p_k-y|} \left(\frac{r}{|p_k|}\right)^\phi \cos(\phi \theta) \left(16^{-k\phi} \left(\cos(\phi\theta)\right)^\phi \right)^{-\gamma} dy\\\\
& +  \int_{H_{2, j}}\ln \frac{8|p_k|}{|p_k-y|} \left(\frac{r}{|p_k|}\right)^\phi \cos(\phi \theta) \left(16^{-k\phi} a_j \cos(\phi\theta) \right)^{-\gamma} dy\\
\sim & 16^{k\phi \gamma+(k-j)\phi} \int_{H_{1, j}} \left(\cos(\phi\theta)\right)^{1-\phi \gamma} dy 
+16^{k\phi \gamma+(k-j)\phi}  a_j^{-\gamma} \int_{H_{2, j}} \left(\cos(\phi\theta)\right)^{1- \gamma} dy \\
\sim & 16^{k\phi \gamma+(k-j)\phi} \cdot16^{-2j}\int_{0}^{a_{j}^{\frac{1+\gamma}{1-\gamma}}} t^{1-\phi \gamma} dt 
+16^{k\phi \gamma+(k-j)\phi}  a_j^{-\gamma} 16^{-2j}\int_{a_{j}^{\frac{1+\gamma}{1-\gamma}}}^{1} t^{1- \gamma} dt \\
\sim  & 
    \begin{cases}
    16^{2k-(\phi+2)j} a_j^{\frac{2}{1-\gamma}}+ 16^{2k-(\phi+2)j} a_j^{-\gamma}, \quad &\mbox{if}\,\, 1<\gamma <2,\\
    16^{2k-(\phi+2)j} a_j^{-2}+
    16^{2k-(\phi+2)j} a_j^{-2} \ln a_j , \quad &\mbox{if} \,\,\gamma =2,\\
    16^{2k-(\phi+2)j} a_j^{\frac{2}{1-\gamma}}+ 16^{2k-(\phi+2)j} a_j^{\frac{2}{1-\gamma}}, \quad &\mbox{if}\,\, \gamma >2.
    \end{cases}\\
    \sim  & 
    \begin{cases}
    16^{2k-(\phi+2)j} a_j^{-\gamma}, \quad &\mbox{if}\,\, 1<\gamma <2,\\
    16^{2k-(\phi+2)j} a_j^{-2} \ln a_j , \quad &\mbox{if} \,\,\gamma =2,\\
    16^{2k-(\phi+2)j} a_j^{\frac{2}{1-\gamma}}, \quad &\mbox{if}\,\, \gamma >2.
    \end{cases}
\end{aligned}
\end{eqnarray*}
It follows from  the definition of $I_j$ and $J_2$ that 
\begin{equation}\label{eq. J2 n=2}
\begin{aligned}
J_2=&\sum_{j=k+1}^{\infty} I_j \\
\lesssim &  \begin{cases}
     16^{2k} \sum_{j=k+1}^{\infty} 16^{-(\phi+2)j} 
    a_j^{-\gamma}, \quad &\mbox{if}\,\, 1<\gamma <2,\\
16^{2k}  \sum_{j=k+1}^{\infty} 16^{-(\phi+2)j}  a_j^{-\gamma} \ln a_j , \quad &\mbox{if} \,\,\gamma =2,\\
    16^{2k} \sum_{j=k+1}^{\infty} 16^{-(\phi+2)j} 
    a_j^{\frac{2}{1-\gamma}}, \quad &\mbox{if}\,\, \gamma >2
    \end{cases}\\
    \sim &
    \begin{cases}
    16^{-k\phi}a_k^{-\gamma}, \quad &\mbox{if}\,\, 1<\gamma <2,\\
    16^{-k\phi} a_k^{-2} \ln a_k , \quad &\mbox{if} \,\,\gamma =2,\\
16^{-k\phi}a_k^{\frac{2}{1-\gamma}}, \quad &\mbox{if}\,\, \gamma >2.
    \end{cases}
\end{aligned}
\end{equation}
Here, we have used the monotonicity of $a_{k}$ (recall Lemma~\ref{lem. a_k is increasing}) in the last step.

   \item Step 3: Estimate the term $J_3$ in \eqref{eq. J1-J4 decompose}.

Note  that for $2\leq j \leq k-1$:
\begin{equation*}
    (1-r) \sim 1\mbox{ and }\frac{|p_k|}{r}\sim 16 ^{k-j}\mbox{ in }H_{1, j}\cup H_{2, j},
\end{equation*}
then by Lemma~\ref{lem. green function for large y} and Example~\ref{ex. growth rate in annuli n=2} we have 
\begin{eqnarray*}
\begin{aligned}
I_j=&\int_{\mathbb{A}_j} G(p_k, y) U^{-\gamma} (y) dy\\
=& \int_{H_{1, j}} G(p_k, y) U^{-\gamma} (y) dy + \int_{H_{2, j}} G(p_k, y) U^{-\gamma} (y) dy\\
\sim & \int_{H_{1, j}}   \left(\frac{|p_k|}{r}\right)^\phi \cos(\phi \theta) (1-r)\left(16^{-j\phi} \left(\cos(\phi\theta)\right)^\phi \right)^{-\gamma} dy\\
& +  \int_{H_{2, j}} \left(\frac{|p_k|}{r}\right)^\phi \cos(\phi \theta) (1-r)\left(16^{-j\phi} a_j  \cos(\phi\theta) \right)^{-\gamma} dy\\
\sim & 16^{j\phi \gamma+(j-k)\phi} \int_{H_{1, j}} \cos(\phi\theta)^{1-\phi \gamma} dy 
+16^{j\phi \gamma+(j-k)\phi}  a_j^{-\gamma} \int_{H_{2, j}} \left(\cos(\phi\theta)\right)^{1- \gamma} dy \\
\sim & 16^{j\phi \gamma+(j-k)\phi}
\cdot16^{-2j}\int_{0}^{a_j^{\frac{1+\gamma}{1-\gamma}}} t^{1- \phi \gamma} dt+ 16^{j\phi \gamma+(j-k)\phi}a_j^{-\gamma}16^{-2j}\int_{a_j^{\frac{1+\gamma}{1-\gamma}}}^{1} t^{1- \gamma} dt
\\
\sim  & 
    \begin{cases}
    16^{-k\phi} a_j^{\frac{2}{1-\gamma}}+16^{-k\phi} a_j^{-\gamma}, \quad &\mbox{if}\,\, 1<\gamma <2,\\
    16^{-k\phi} a_j^{-2}+16^{-k\phi}
     a_j^{-2} \ln a_j , \quad &\mbox{if} \,\,\gamma =2,\\
    16^{-k\phi} a_j^{\frac{2}{1-\gamma}}+16^{-k\phi} a_j^{\frac{2}{1-\gamma}}, \quad &\mbox{if}\,\, \gamma >2
    \end{cases}\\
    \sim  & 
    \begin{cases}
    16^{-k\phi} a_j^{-\gamma}, \quad &\mbox{if}\,\, 1<\gamma <2,\\
    16^{-k\phi}
     a_j^{-2} \ln a_j, \quad &\mbox{if} \,\,\gamma =2,\\
    16^{-k\phi} a_j^{\frac{2}{1-\gamma}}, \quad &\mbox{if}\,\, \gamma >2,
    \end{cases}
\end{aligned}
\end{eqnarray*}
where we have used $\phi(\gamma+1)-2=0$ in the computation. Therefore, from the definition of $J_3$ and $I_j$, one has 
\begin{eqnarray}\label{eq. J3 n=2}
    \begin{aligned}
     J_3=& \sum_{j=2}^{k-1} I_j 
     \sim  \begin{cases}
    16^{-k\phi} \sum_{j=2}^{k-1}  a_j^{-\gamma}, \quad &\mbox{if}\,\, 1<\gamma <2,\\
    16^{-k\phi} \sum_{j=2}^{k-1}
    
     a_j^{-2} \ln a_j , \quad &\mbox{if} \,\,\gamma =2, \\
   16^{-k\phi} \sum_{j=2}^{k-1}  a_j^{\frac{2}{1-\gamma}}, \quad &\mbox{if}\,\, \gamma >2.
     \end{cases}
    \end{aligned}
\end{eqnarray}

  \item Step 4: Estimate the term $J_4$ in \eqref{eq. J1-J4 decompose}.

In $\mathbb{A}_1$, since $\frac{|p_k|}{r} \sim 16^{-k}$, and $U(y) \sim \left(dist(y, \partial \Omega)\right)^{\frac{2}{1+\gamma}}  $ by \eqref{eq. estimate in A1}, then by Lemma~\ref{lem. green function for large y} we have

  \begin{eqnarray}\label{eq. J4 n=2}
    \begin{aligned}
    J_4=&\int_{\mathbb{A}_1} G(p_k, y) U^{-\gamma} (y) dy\\
    \sim & \int_{\mathbb{A}_1} \left(\frac{|p_k|}{r}\right)^\phi \cos(\phi \theta) (1-r) \left(dist  (y, \partial \Omega)\right) ^\frac{-2\gamma}{1+\gamma}dy\\
    \sim & 16^{-k\phi} \int_{\mathbb{A}_1}  \left(dist (y, \partial \Omega)\right)^{1-\frac{2\gamma}{1+\gamma}}dy\\
    \sim & \begin{cases}
   16^{-k\phi} a_1^{-\gamma}, \quad &\mbox{if}\,\, 1<\gamma <2,\\
      16^{-k\phi}
     a_1^{-2} \ln a_1, \quad &\mbox{if} \,\,\gamma =2, \\
   16^{-k\phi}a_1^{\frac{2}{1-\gamma}}, \quad &\mbox{if}\,\, \gamma >2,
     \end{cases}
    \end{aligned}
    \end{eqnarray}
where we have used the fact $1-\frac{2\gamma}{1+\gamma}>-1$. 
\item Step 5: We collect the estimates \eqref{eq. J1 n=2},\eqref{eq. J2 n=2},\eqref{eq. J3 n=2},\eqref{eq. J4 n=2}, and then obtain
\begin{equation*}
    U(p_{k})\sim \left\{\begin{aligned}
   &16^{-k\phi} \sum_{j=1}^{k} a_j^{-\gamma}, &\mbox{if }&1<\gamma <2,\\
      &16^{-k\phi} \sum_{j=1}^{k}a_j^{-2} \ln a_j,&\mbox{if }&\gamma =2, \\
   &16^{-k\phi} \sum_{j=1}^{k} a_j^{\frac{2}{1-\gamma}},&\mbox{if }&\gamma >2.
     \end{aligned}\right.
\end{equation*}
Recalling that $U(p_{k})=16^{-k\phi}a_{k}$ from \eqref{eq. a_k definition}, we have  thus complete the proof of  Lemma~\ref{lem. discrete integral equation} in the case $n=2$.

\end{itemize}

\subsection{Case \texorpdfstring{$n=3$}{Lg}}
Now we prove Lemma~\ref{lem. discrete integral equation} in the case $n\geq3$. We adopt the same  notations in $n=2$, including:
\begin{equation*}
    \Omega_{1}-\Omega_{4},J_{1}-J_{4},J_{11}-J_{13},\mbox{ and }I_{j}\mbox{'s}
\end{equation*}
here. The computation here differs from the one for $n=2$, mainly because the Green function takes a different form in $n\geq3$. In addition, the scaling in higher  dimensions produces different scaling factors in the volume form.

 \begin{itemize}
        \item Step 1: Estimate the term $J_1$ in \eqref{eq. J1-J4 decompose}.
 
        Noting that $p_k \in \Omega_1$ is a singular point,  then we compute the integral in $\mathbb{A}_k \cup B_{\varepsilon_k}(p_k)$ and 
     $B_{\varepsilon_k}(p_k)$ respectively, where $\varepsilon>0$ is a uniform and sufficiently small  constant, and  $\varepsilon_k= \varepsilon \cdot 16^{-k}$. 

     Let $y=r\vec{\theta}$, and let $\eta(y)=\eta_{\Sigma}(\vec{\theta})$ be the defining function constructed in Definition~\ref{def. defining function eta}. We denote 
     $$
     D_{1, k}= \Big(\mathbb{A}_k \backslash B_{\varepsilon_k}(p_k)\Big) \cap \Big\{\eta(y)\leq a_k^{\frac{1+\gamma}{1-\gamma}}\Big\}
     $$
     and 
     $$
     D_{2, k}= \Big(\mathbb{A}_k \backslash B_{\varepsilon_k}(p_k)\Big) \cap \Big\{\eta(y)> a_k^{\frac{1+\gamma}{1-\gamma}}\Big\}.
     $$
As in the case $n=2$, here, in Figure 4,  $D_{1,k}$ still denotes the gray shaded region within $\Omega_1$,  while  $D_{2,k}$ still represents the portion of  $\Omega_1$ that excludes both the shadowed region and the ball $B_{\varepsilon_k}(p_k)$.
\begin{eqnarray}\label{eq. J1 decomposition n>2}
\begin{aligned}
    J_1=& \int_{\mathbb{A}_k} G(p_k, y) U^{-\gamma} (y)dy \\
    =& \int_{D_{1, k}} G(p_k, y) U^{-\gamma} (y)dy +\int_{D_{2, k}} G(p_k, y) U^{-\gamma} (y)dy +\int_{B_{\varepsilon_k}(p_k)} G(p_k, y) U^{-\gamma} (y)dy \\
   =:& J_{11}+J_{12}+J_{13}.
\end{aligned}
\end{eqnarray}

Since
\begin{equation*}
    2|p_k|=16^{1-k}=r_k>r_{k+1}\mbox{ and }|p_{k}-y|^{2-n}\sim 16^{k(n-2)}\mbox{ in }D_{1, k}\cup D_{2, k},
\end{equation*}
and that $\frac{r}{|p_k|}\sim 1$ in $\mathbb{A}_k$, then by Lemma~\ref{lem. green function for small y} and Corollary~\ref{cor. estimate in a cone}, we obtain  
\begin{eqnarray}\label{eq. J11 n>2}
\begin{aligned}
    J_{11}=& \int_{D_{1, k}} G(p_k, y) U^{-\gamma} (y)dy \\
    \sim &  \int_{D_{1, k}} |p_{k}-y|^{2-n} \left(\frac{r}{|p_k|}\right)^\phi \eta(y) \left(16^{-k\phi} \eta^\phi (y)\right)^{-\gamma} dy\\
    \sim & 16^{k(n-2+\phi\gamma)} \int_{D_{1, k}} \eta^{1-\phi\gamma} (y)dy\\
    \sim & 16^{k(n-2+\phi\gamma)} \cdot16^{-kn}\int_{0}^{a_k^{\frac{1+\gamma}{1-\gamma}}} t^{1-\phi\gamma} dt \\
    \sim & 16^{-k\phi} a_k^{\frac{2}{1-\gamma}},
\end{aligned}
\end{eqnarray}
where we have use the fact that $\phi=\frac{2}{1+\gamma}$ and $1-\phi \gamma=\frac{1-\gamma}{1+\gamma}\in (-1, 0)$. Besides, Lemma~\ref{lem. definiting function in a patch} (e)(f) is also used in the following estimate, which was used in the computation above:
\begin{equation*}
    \int_{D_{1, k}} \eta^{1-\phi\gamma} (y)dy\sim16^{-kn}\int_{0}^{a_k^{\frac{1+\gamma}{1-\gamma}}} t^{1-\phi\gamma} dt.
\end{equation*}

For the term $J_{12}$,  by Lemma~\ref{lem. green function for small y} and Corollary~\ref{cor. estimate in a cone}, we have 
\begin{eqnarray}\label{eq. J12 n>2}
\begin{aligned}
    J_{12}=& \int_{D_{2, k}} G(p_k, y) U^{-\gamma} (y)dy \\
    \sim &  \int_{D_{2, k}} |p_{k}-y|^{2-n} \left(\frac{r}{|p_k|}\right)^\phi \eta(y) \left(16^{-k\phi} a_k \eta(y) \right)^{-\gamma} dy\\
    \sim & 16^{k(n-2+\phi \gamma)} a_k^{-\gamma}\int_{D_{2, k}} \eta^{1-\gamma}(y) dy\\
    \leq & 16^{k(n-2+\phi \gamma)} a_k^{-\gamma} 16^{-kn}\int_{a_k^{\frac{1+\gamma}{1-\gamma}}}^{1} t^{1-\gamma} dt \\
    \sim & 
    \begin{cases}
    16^{-k\phi} a_k^{-\gamma}, \quad &\mbox{if} \,\, 1<\gamma<2,\\
    16^{-k\phi} a_k^{-2} \ln a_k, \quad &\mbox{if} \,\, \gamma =2,\\
     16^{-k\phi} a_k^{\frac{2}{1-\gamma}}, \quad &\mbox{if} \,\, \gamma >2.
    \end{cases}
\end{aligned}
\end{eqnarray}
In the computation above, the estimate
\begin{equation*}
    \int_{D_{2, k}} \eta^{1-\gamma} (y) dy\sim16^{-kn}\int_{a_k^{\frac{1+\gamma}{1-\gamma}}}^{1} t^{1-\gamma} dt
\end{equation*}
also follows from Lemma~\ref{lem. definiting function in a patch} (e)(f).

For the last term $J_{13}$  in \eqref{eq. J1-J4 decompose}, since  $\eta(y)\sim 1$ in $B_{\varepsilon_k}(p_k)$, one can derive from Lemma~\ref{lem. green function for small y} and Corollary~\ref{cor. estimate in a cone} that 
\begin{eqnarray}\label{eq. J13 n>2}
\begin{aligned}
    J_{13}=& \int_{B_{\varepsilon_k}(p_k)} G(p_k, y) U^{-\gamma} (y)dy \\
    =& \int_{B_{\varepsilon_k}(p_k)} |p_{k}-y|^{2-n} \left(\frac{r}{|p_k|}\right)^\phi \eta(y) \left(16^{-k\phi} a_k \eta(y) \right)^{-\gamma} dy\\
    \sim &  16^{k\phi \gamma} a_k^{-\gamma}\int_{B_{\varepsilon_k}(p_k)} |p_{k}-y|^{2-n} dy\\
     \sim &  16^{k\phi \gamma} a_k^{-\gamma} \int_{\mathbb{S}^{n-1}}\int_0^{\varepsilon_k} \rho^{2-n} \rho^{n-1}d\rho d\theta\\
     \sim & 16^{k\phi \gamma} a_k^{-\gamma} 16^{-2k} \int_0^{\varepsilon} \rho_1 d\rho_1\quad(\mbox{here }\rho_{1}:=16^{k}\rho)\\
     \sim & 16^{-k\phi}   a_k^{-\gamma}.
      \end{aligned}
\end{eqnarray}

Combining \eqref{eq. J1 decomposition n>2},\eqref{eq. J11 n>2},\eqref{eq. J12 n>2},\eqref{eq. J13 n>2} yields 
\begin{equation}\label{eq. J1 n>2}
    J_1\sim 
    \begin{cases}
    16^{-k\phi} a_k^{-\gamma}, \quad &\mbox{if} \,\, 1<\gamma<2,\\
    16^{-k\phi} a_k^{-2} \ln a_k, \quad &\mbox{if} \,\, \gamma =2,\\
     16^{-k\phi} a_k^{\frac{2}{1-\gamma}}, \quad &\mbox{if}\,\, \gamma >2.
    \end{cases}
\end{equation}

   \item Step 2: Estimate the term $J_2$ in \eqref{eq. J1-J4 decompose}.

Denote a shadowed region
 $$
     H_{1, j}= \mathbb{A}_j \cap \{\eta(y)\leq a_k^{\frac{1+\gamma}{1-\gamma}}\}
     $$
     and an unshadowed region
     $$
     H_{2, j}= \mathbb{A}_j  \cap \{\eta(y)> a_k^{\frac{1+\gamma}{1-\gamma}}\}.
     $$
Let $y=r\vec{\theta}\in \mathbb{A}_j$ for $j \geq k+1$. Since
\begin{equation*}
    |p_k-y|^{2-n}\sim 16^{k(n-2)}\mbox{ and }\frac{r}{|p_k|}\sim 16^{k-j}\quad\mbox{in }H_{1, j}\cup H_{2, j},
\end{equation*}
then by Lemma~\ref{lem. green function for small y} and Corollary~\ref{cor. estimate in a cone}, we have 
\begin{eqnarray*}
\begin{aligned}
I_j=&\int_{\mathbb{A}_j} G(p_k, y) U^{-\gamma} (y) dy\\
=& \int_{H_{1, j}} G(p_k, y) U^{-\gamma} (y) dy + \int_{H_{2, j}} G(p_k, y) U^{-\gamma} (y) dy\\
\sim & \int_{H_{1, j}}  |p_k-y|^{2-n} \left(\frac{r}{|p_k|}\right)^\phi \eta(y) \left(16^{-k\phi} \eta^\phi (y) \right)^{-\gamma} dy\\
& +  \int_{H_{2, j}}|p_k-y|^{2-n}\left(\frac{r}{|p_k|}\right)^\phi \eta(y) \left(16^{-k\phi} a_j \eta(y) \right)^{-\gamma} dy\\
\sim & 16^{k(n-2+\phi \gamma)+(k-j)\phi} \int_{H_{1, j}} \eta^{1-\phi \gamma} (y)dy \\
&+16^{k(n-2+\phi \gamma)+(k-j)\phi}  a_j^{-\gamma} \int_{H_{2, j}} \eta^{1- \gamma}(y) dy \\
\sim & 16^{k(n-2+\phi \gamma)+(k-j)\phi} \cdot16^{-jn}\int_{0}^{a_{k}^{\frac{1+\gamma}{1-\gamma}}} t^{1-\phi \gamma} dt \\
&+16^{k(n-2+\phi \gamma)+(k-j)\phi}  a_j^{-\gamma} 16^{-jn}\int_{a_{k}^{\frac{1+\gamma}{1-\gamma}}}^{1} t^{1- \gamma} dt \\
\sim  & 
    \begin{cases}16^{(k-j)(n+\phi)-k\phi} a_j^{\frac{2}{1-\gamma}}+
    16^{(k-j)(n+\phi)-k\phi} a_j^{-\gamma} , \quad &\mbox{if} \,\,1<\gamma<2,\\
    16^{(k-j)(n+\phi)-k\phi} a_j^{-2}+
    16^{(k-j)(n+\phi)-k\phi} a_j^{-2} \ln a_j , \quad &\mbox{if} \,\,\gamma =2,\\
    16^{(k-j)(n+\phi)-k\phi} a_j^{\frac{2}{1-\gamma}}+ 16^{(k-j)(n+\phi)-k\phi} a_j^{\frac{2}{1-\gamma}}, \quad &\mbox{if}\,\, \gamma >2
    \end{cases}\\
    \sim  & 
    \begin{cases}16^{kn-(n+\phi)j} a_j^{-\gamma} , \quad &\mbox{if} \,\,1<\gamma<2,\\
    16^{kn-(n+\phi)j} a_j^{-2} \ln a_j , \quad &\mbox{if} \,\,\gamma =2,\\
    16^{kn-(n+\phi)j} a_j^{\frac{2}{1-\gamma}}, \quad &\mbox{if}\,\, \gamma >2.
    \end{cases}
\end{aligned}
\end{eqnarray*}
It follows from the monotonicity of $a_{j}$ (Lemma~\ref{lem. a_k is increasing}) and from the definition of $I_j$ and $J_2$ that 
\begin{equation}\label{eq. J2 n>2}
\begin{aligned}
J_2=&\sum_{j=k+1}^{\infty} I_j \\
\sim &  \begin{cases}16^{kn}  \sum_{j=k+1}^{\infty} 16^{-(n+\phi)j}  a_j^{-\gamma} , \quad &\mbox{if} \,\,1<\gamma<2,\\
16^{kn}  \sum_{j=k+1}^{\infty} 16^{-(n+\phi)j}  a_j^{-2} \ln a_j , \quad &\mbox{if} \,\,\gamma =2,\\
    16^{kn} \sum_{j=k+1}^{\infty} 16^{-(n+\phi)j} 
    a_j^{\frac{2}{1-\gamma}}, \quad &\mbox{if}\,\, \gamma >2
    \end{cases}\\
    \lesssim &
    \begin{cases}16^{-k\phi} a_k^{-\gamma} , \quad &\mbox{if} \,\,1<\gamma<2,\\
    16^{-k\phi} a_k^{-2} \ln a_k , \quad &\mbox{if} \,\,\gamma =2,\\
16^{-k\phi}a_k^{\frac{2}{1-\gamma}}, \quad &\mbox{if}\,\, \gamma >2.
    \end{cases}
\end{aligned}
\end{equation}

   \item Step 3: Estimate the term $J_3$ in \eqref{eq. J1-J4 decompose}.

Note  that for $2\leq j \leq k-1$:
\begin{equation*}
    (1-r) \sim 1\mbox{ and }\frac{|p_k|^{\phi}}{r^{\phi+n-2}}\sim 16 ^{(j-k)\phi+(n-2)j}\quad\mbox{in }H_{1, j}\cup H_{2, j},
\end{equation*}
then by Lemma~\ref{lem. green function for large y} and Corollary~\ref{cor. estimate in a cone}, we have 
\begin{eqnarray*}
\begin{aligned}
I_j=&\int_{\mathbb{A}_j} G(p_k, y) U^{-\gamma} (y) dy\\
=& \int_{H_{1, j}} G(p_k, y) U^{-\gamma} (y) dy + \int_{H_{2, j}} G(p_k, y) U^{-\gamma} (y) dy\\
\sim & \int_{H_{1, j}}   \frac{|p_k|^{\phi}}{r^{\phi+n-2}}(1-r)\eta(y)\cdot\left(16^{-j\phi} \eta^\phi(y) \right)^{-\gamma} dy\\
& +  \int_{H_{2, j}} \frac{|p_k|^{\phi}}{r^{\phi+n-2}}(1-r)\eta(y)\cdot\left(16^{-j\phi} a_j  \eta(y) \right)^{-\gamma} dy\\
\sim & 16^{nj-k\phi} \int_{H_{1, j}} \eta^{1-\phi \gamma} (y) dy+16^{nj-k\phi}  a_j^{-\gamma} \int_{H_{2, j}} \eta^{1- \gamma} (y) dy \\
\sim & 16^{nj-k\phi}\cdot 16^{-nj}\int_{0}^{a_{k}^{\frac{1+\gamma}{1-\gamma}}} t^{1-\phi \gamma} dt+16^{nj-k\phi}  a_j^{-\gamma} 16^{-nj}\int_{a_{k}^{\frac{1+\gamma}{1-\gamma}}}^{1} t^{1-\gamma} dt \\
\sim  & 
    \begin{cases}16^{-k\phi} a_j^{\frac{2}{1-\gamma}}+16^{-k\phi}
     a_j^{-\gamma} , \quad &\mbox{if} \,\,1<\gamma<2,\\
     16^{-k\phi} a_j^{-2}+16^{-k\phi}
     a_j^{-2} \ln a_j , \quad &\mbox{if} \,\,\gamma =2,\\
    16^{-k\phi} a_j^{\frac{2}{1-\gamma}}+16^{-k\phi} a_j^{\frac{2}{1-\gamma}}, \quad &\mbox{if}\,\, \gamma >2
    \end{cases}\\
    \sim  & 
    \begin{cases}16^{-k\phi}
     a_j^{-\gamma} , \quad &\mbox{if} \,\,1<\gamma<2,\\
     16^{-k\phi}
     a_j^{-2} \ln a_j , \quad &\mbox{if} \,\,\gamma =2,\\
    16^{-k\phi} a_j^{\frac{2}{1-\gamma}}, \quad &\mbox{if}\,\, \gamma >2.
    \end{cases}
\end{aligned}
\end{eqnarray*}
Therefore, from the definition of $J_3$ and $I_j$, one has 
\begin{eqnarray}\label{eq. J3 n>2}
    \begin{aligned}
     J_3=& \sum_{j=2}^{k-1} I_j 
     \sim  \begin{cases}
      16^{-k\phi} \sum_{j=2}^{k-1}
     a_j^{-\gamma}, \quad &\mbox{if} \,\,1<\gamma<2, \\
      16^{-k\phi} \sum_{j=2}^{k-1}
     a_j^{-2} \ln a_j, \quad &\mbox{if} \,\,\gamma =2, \\
   16^{-k\phi} \sum_{j=2}^{k-1} a_j^{\frac{2}{1-\gamma}}, \quad &\mbox{if}\,\, \gamma >2.
     \end{cases}
    \end{aligned}
\end{eqnarray}

  \item Step 4: Estimate the term $J_4$ in \eqref{eq. J1-J4 decompose}.

In $\mathbb{A}_1$, since $\frac{|p_k|^{\phi}}{r^{\phi+n-2}} \sim 16^{-k\phi}$  and that $U(x)\sim \left(dist(x,\partial\Omega)\right)^{\frac{2}{1+\gamma}}$ by \eqref{eq. estimate in A1}, then by Lemma~\ref{lem. green function for large y} we have

  \begin{eqnarray}\label{eq. J4 n>2}
    \begin{aligned}
    J_4=&\int_{\mathbb{A}_1} G(p_k, y) U^{-\gamma} (y) dy\\
    \sim & \int_{\mathbb{A}_1} \frac{|p_k|^{\phi}}{r^{\phi+n-2}} \eta(y) (1-r) \left(dist (y, \partial \Omega) \right) ^{-\frac{2\gamma}{1+\gamma}}dy\\
    \sim & 16^{-k\phi} \int_{\mathbb{A}_1}  \left(dist (y, \partial \Omega)\right)  ^{1-\frac{2\gamma}{1+\gamma}}dy\\
    \sim & \begin{cases}
   16^{-k\phi} a_1^{-\gamma}, \quad &\mbox{if}\,\, 1<\gamma <2,\\
      16^{-k\phi}
     a_1^{-2} \ln a_1, \quad &\mbox{if} \,\,\gamma =2, \\
   16^{-k\phi}a_1^{\frac{2}{1-\gamma}}, \quad &\mbox{if}\,\, \gamma >2,
     \end{cases}
    \end{aligned}
    \end{eqnarray}
where we have used the fact $1-\frac{2\gamma}{1+\gamma}>-1$. 

\item Step 5: We collect the estimates \eqref{eq. J1 n>2},\eqref{eq. J2 n>2},\eqref{eq. J3 n>2},\eqref{eq. J4 n>2}, and then obtain exactly the same equation as in the case $n=2$. Therefore, we have proven Lemma~\ref{lem. discrete integral equation} in the case $n\geq3$.

\end{itemize}

\section{Proof of main results}\label{sec. proof of main results}
In this section, we prove our main results. We begin with the  proof of Theorem~\ref{thm. asymptotic estimate}.
\begin{proof}[Proof of Theorem~\ref{thm. asymptotic estimate}]
    It suffices to show the following asymptotic behavior of $a_{k}$:
    \begin{equation}\label{eq. a_k growth}
        a_{k}\sim\left\{\begin{aligned}
            &k^{\frac{1}{1+\gamma}},&\mbox{if }&1<\gamma<2,\\
            &(k\ln{k})^{\frac{1}{3}},&\mbox{if }&\gamma=2,\\
            &k^{\frac{\gamma-1}{1+\gamma}},&\mbox{if }&\gamma>2.
        \end{aligned}\right.
    \end{equation}
   Once \eqref{eq. a_k growth} is shown, we apply Corollary~\ref{cor. estimate in a cone} to each annulus $\mathbb{A}_{k}$ for $k\geq2$, and obtain the desired growth rate in $\Omega\cap B_{1/16}$. The estimate in $\Omega\setminus B_{1/16}$ was already established  in \eqref{eq. estimate in A1}. Therefore, we have proven Theorem~\ref{thm. asymptotic estimate}.

    It remains to show \eqref{eq. a_k growth} using the discrete integral equation obtained in Lemma~\ref{lem. discrete integral equation}. For each $k\geq1$, define
    \begin{equation*}
        S_{k}=\left\{\begin{aligned}
            &\sum_{j=1}^{k}a_{j}^{-\gamma},&\mbox{if }&1<\gamma<2,\\
            &\sum_{j=1}^{k}a_{j}^{-2}\ln{a_{j}},&\mbox{if }&\gamma=2,\\
            &\sum_{j=1}^{k}a_{j}^{\frac{2}{1-\gamma}},&\mbox{if }&\gamma>2.
        \end{aligned}\right.
    \end{equation*}
   In addition, let $S_{0}:=0.$
    
    Then for any $k\geq1$, we obtain 
    \begin{equation*}
        S_{k}-S_{k-1}=\left\{\begin{aligned}
            &a_{k}^{-\gamma},&\mbox{if }&1<\gamma<2.\\
            &a_{k}^{-2}\ln{a_{k}},&\mbox{if }&\gamma=2.\\
            &a_{k}^{\frac{2}{1-\gamma}},&\mbox{if }&\gamma>2.
        \end{aligned}\right.
    \end{equation*}
    From Lemma~\ref{lem. discrete integral equation}, we have $a_{k}\sim S_{k}$, so we obtain a discrete ODE for $S_{k}$:
    \begin{equation}\label{eq. S_k ODE}
        S_{k}-S_{k-1}\sim\left\{\begin{aligned}
            &S_{k}^{-\gamma},&\mbox{if }&1<\gamma<2.\\
            &S_{k}^{-2}\ln{S_{k}},&\mbox{if }&\gamma=2.\\
            &S_{k}^{\frac{2}{1-\gamma}},&\mbox{if }&\gamma>2.
        \end{aligned}\right.
    \end{equation}
    Analogous to  the ODE technique, we let
    \begin{equation*}
        T_{k}=\left\{\begin{aligned}
            &S_{k}^{\gamma+1},&\mbox{if }&1<\gamma<2,\\
            &\frac{S_{k}^{3}}{\ln{S_{k}}},&\mbox{if }&\gamma=2,\\
            &S_{k}^{\frac{\gamma+1}{\gamma-1}},&\mbox{if }&\gamma>2.
        \end{aligned}\right.
    \end{equation*}
    Recall that from Lemma~\ref{lem. a_k is increasing}, $a_{1}\sim1$ and $a_{k}$ is increasing. Consequently, it is easy to verify that adjacent terms of $S_{k}$ are equivalent, that is, 
    \begin{equation*}
        S_{k+1}\sim S_{k},\quad\mbox{for all }k\geq1.
    \end{equation*}
    We then take the ``derivative'' of $T_{k}$, and by the ``chain rule'', we have
    \begin{equation*}
        \frac{T_{k+1}-T_{k}}{S_{k+1}-S_{k}}\sim\left\{\begin{aligned}
            &(\gamma+1)S_{k}^{\gamma},&\mbox{if }&1<\gamma<2,\\
            &S_{k}^{2}\cdot\frac{3\ln{S_{k}}-1}{(\ln{S_{k}})^{2}},&\mbox{if }&\gamma=2,\\
            &\frac{\gamma+1}{\gamma-1}S_{k}^{\frac{2}{\gamma-1}},&\mbox{if }&\gamma>2
        \end{aligned}\right.\sim\left\{\begin{aligned}
            &S_{k}^{\gamma},&\mbox{if }&1<\gamma<2,\\
            &\frac{S_{k}^{2}}{\ln{S_{k}}},&\mbox{if }&\gamma=2,\\
            &S_{k}^{\frac{2}{\gamma-1}},&\mbox{if }&\gamma>2.
        \end{aligned}\right.
    \end{equation*}
    By \eqref{eq. S_k ODE}, we obtain that $T_{k+1}-T_{k}\sim1$, and thus $T_{k}\sim k$. From the construction of $T_{k}$, we have
    \begin{equation*}
        a_{k}\sim S_{k}\sim\left\{\begin{aligned}
            &T_{k}^{\frac{1}{\gamma+1}},&\mbox{if }&1<\gamma<2,\\
            &(T_{k}\ln{T_{k}})^{\frac{1}{3}},&\mbox{if }&\gamma=2,\\
            &T_{k}^{\frac{\gamma-1}{\gamma+1}},&\mbox{if }&\gamma>2
        \end{aligned}\right.\sim\left\{\begin{aligned}
            &k^{\frac{1}{1+\gamma}},&\mbox{if }&1<\gamma<2,\\
            &(k\ln{k})^{\frac{1}{3}},&\mbox{if }&\gamma=2,\\
            &k^{\frac{\gamma-1}{1+\gamma}},&\mbox{if }&\gamma>2,
        \end{aligned}\right.
    \end{equation*}
    hence proving \eqref{eq. a_k growth}.
\end{proof}

Next, we prove Theorem~\ref{thm. modulus of continuity}.
\begin{proof}[Proof of Theorem~\ref{thm. modulus of continuity}]
    We first prove Theorem~\ref{thm. modulus of continuity} (1). We let $U_{f}$ be the solution to the following problem:
    \begin{equation*}
        \begin{cases}
            -\Delta U_{f}(x)=f(x)\cdot U_{f}(x)^{-\gamma}&\mbox{in }\,\,\Omega,\\
            U_{f}(x)=0&\mbox{on }\,\,\partial\Omega.
        \end{cases}
    \end{equation*}
    We also denote $U_{\lambda}$ and $U_{\Lambda}$ as the solution $U_{f}$ when $f\equiv\lambda$ and $f\equiv\Lambda$, respectively. It then follows from the maximum  principle that $U_{\lambda}\leq U_{f}\leq U_{\Lambda}$ if $\lambda\leq f(x)\leq\Lambda$. Moreover, $U_{\lambda}$ and $U_{\Lambda}$ are proportional to $U(x)$ solving \eqref{eq. special solution U(x)}. Therefore, we have
    \begin{equation}\label{eq. U_f bound}
        c(\lambda,\gamma)\cdot U(x)\leq U_{f}(x)\leq C(\Lambda,\gamma)\cdot U(x).
    \end{equation}
    We then apply \cite[Theorem 1.4]{GuLiZh25} for the pair $(u,U_{f})$ and obtain that
    \begin{equation*}
        C^{-1}(n,L,\gamma,\lambda,\Lambda)\min\Big\{\frac{\|u\|_{L^{\infty}(\mathcal{GC}_{3R}(0))}}{\|U_{f}\|_{L^{\infty}(\mathcal{GC}_{3R}(0))}},1\Big\}\leq\frac{u}{U_{f}}\leq C(n,L,\gamma,\lambda,\Lambda)\max\Big\{\frac{\|u\|_{L^{\infty}(\mathcal{GC}_{3R}(0))}}{\|U_{f}\|_{L^{\infty}(\mathcal{GC}_{3R}(0))}},1\Big\}
    \end{equation*}
    in $\Omega\cap B_{1/16}$. Since
    \begin{equation*}
        C^{-1}(n,L,\gamma,\lambda,\Lambda)\leq\|U_{f}\|_{L^{\infty}(\mathcal{GC}_{3R}(0))}\leq C(n,L,\gamma,\lambda,\Lambda)
    \end{equation*}
    and
    $\|u\|_{L^{\infty}(\mathcal{GC}_{3R}(0))}\geq C^{-1}(n,L,\gamma,\lambda,\Lambda)$, we conclude that
    \begin{equation*}
        C^{-1}(n,L,\gamma,\lambda,\Lambda)\leq\frac{u}{U_{f}}\leq C(n,L,\gamma,\lambda,\Lambda)\frac{\|u\|_{L^{\infty}(\mathcal{GC}_{3R}(0))}}{\|U_{f}\|_{L^{\infty}(\mathcal{GC}_{3R}(0))}}\leq C(n,L,\gamma,\lambda,\Lambda)\|u\|_{L^{\infty}(\mathcal{GC}_{3R}(0))}
    \end{equation*}
    in $\Omega\cap B_{1/16}$. Together with the bound \eqref{eq. U_f bound}, we have proven Theorem~\ref{thm. modulus of continuity} (1).

    Next, let's prove Theorem~\ref{thm. modulus of continuity} (2). By Lemma~\ref{lem. review : interior lower bound}, we know that
    \begin{equation*}
        u(x)\geq c(n,L,\gamma,\lambda)\cdot \left(dist(x,\partial\Omega)\right)^{\frac{2}{1+\gamma}},
    \end{equation*}
    so we have
    \begin{equation}\label{eq. RHS bound}
        |\Delta u|\leq C\cdot \left(dist(x,\partial\Omega)\right)^{-\frac{2\gamma}{1+\gamma}}
    \end{equation}
    with $C=C(n,L,\gamma,\lambda).$
    
    Let $r\leq\frac{1}{4}$. We infer  from Theorem~\ref{thm. modulus of continuity} (1) that,  in the annulus $\Omega\cap(B_{2r}\setminus B_{r/4})$, we have 
    \begin{equation*}
        u(x)\leq C\|u\|_{L^{\infty}(\Omega)}\cdot\sigma(r)\cdot\eta(x)^{\frac{2}{1+\gamma}},\quad\sigma(t)=\left\{\begin{aligned}
                &t^{\frac{2}{1+\gamma}}(\ln{\frac{1}{t}})^{\frac{1}{1+\gamma}},&\mbox{if }&1<\gamma<2,\\
                &t^{\frac{2}{3}}(\ln{\frac{1}{t}})^{\frac{1}{3}}(\ln{\ln{\frac{1}{t}}})^{\frac{1}{3}},&\mbox{if }&\gamma=2,\\
                &t^{\frac{2}{1+\gamma}}(\ln{\frac{1}{t}})^{\frac{\gamma-1}{1+\gamma}},&\mbox{if }&\gamma>2.
            \end{aligned}\right.
    \end{equation*}
    Since we can infer from \eqref{eq. RHS bound} that $|\Delta u|\leq C\cdot r^{-\frac{2\gamma}{1+\gamma}}\cdot\eta(x)^{-\frac{2\gamma}{1+\gamma}}$ in $\Omega\cap(B_{2r}\setminus B_{r/4})$, we can apply the Schauder estimate with singular right-hand side (see for example \cite[Proposition 3.4]{GuLi93}) and conclude that
    \begin{equation}\label{eq. continuity in one annulus}
        |u(x)-u(y)|\leq C\|u\|_{L^{\infty}(\Omega)}\cdot\sigma(|x-y|),\quad\mbox{for any }x,y\in\Omega\cap(B_{r}\setminus B_{r/2}).
    \end{equation}
    On the other hand, if $x,y$ do not belong to the same annulus $\Omega\cap(B_{r}\setminus B_{r/2})$, then without loss of generality, we assume $|x|\geq2|y|$. We first see that $|x-y|\sim|x|$. Next, we insert $N$ points $q_{1},q_{2},\cdots,q_{N}$ between $x:=q_{0}$ and $y:=q_{N+1}$ such that
    \begin{itemize}
        \item $\displaystyle N\sim\log_{2}(\frac{|x|}{|y|})$;
        \item Adjacent points $q_{i}$ and $q_{i+1}$ belong to the same annulus $\Omega\cap(B_{r_{i}}\setminus B_{r_{i}/2})$, here $i\in\{0,\cdots,N\}$.
    \end{itemize}
    By \eqref{eq. continuity in one annulus}, we have
    \begin{equation*}
        |u(q_{i})-u(q_{i+1})|\leq C\|u\|_{L^{\infty}(\Omega)}\cdot\sigma(1.5^{-i}|x|),\quad\mbox{for }i\in\{0,\cdots,N\}.
    \end{equation*}
    Therefore, we sum up the estimate above and have
    \begin{equation*}
        |u(x)-u(y)|\leq C\|u\|_{L^{\infty}(\Omega)}\sum_{i=0}^{N}\sigma(1.5^{-i}|x|)\sim C\|u\|_{L^{\infty}(\Omega)}\sigma(|x|)\sim C\|u\|_{L^{\infty}(\Omega)}\sigma(|x-y|).
    \end{equation*}
    Finally, the optimality of $\sigma(t)$ follows easily by choosing $x=0$ and $y=t\vec{e_{n}}$. This completes  the proof of Theorem~\ref{thm. modulus of continuity}.
\end{proof}

Theorem~\ref{thm. three equivalent statements} then follows immediately.
\begin{proof}[Proof of Theorem~\ref{thm. three equivalent statements}]
    First, we use \cite[Theorem 1.3]{GuLiZh25} together with the lower solution \eqref{eq. precise lower solution in GLZ} to prove the implication  $(b)\Rightarrow(c)$. In fact, we can even remove the assumption that $\partial Cone_{\Sigma}$ is $C^{1,1}$ away from the origin when proving $(b)\Rightarrow(c)$.

    Second, from Theorem~\ref{thm. asymptotic estimate} or Theorem~\ref{thm. modulus of continuity} (1), we see that $(c)\Rightarrow(a)$.

    It remains to show $(a)\Rightarrow(b)$. The key idea is to construct a continuous super-solution to \eqref{eq. growth rate special condition in critical case}. To this end, we define 
    \begin{equation*}
        \overline{w}(x)=H(x)^{\varepsilon},
    \end{equation*}
    where $\varepsilon=1-\frac{\gamma}{2}$ is sufficiently small, and satisfies  $\varepsilon\in(0,1)$ under the assumption of (a). In fact, we see $\overline{w}(x)$ is continuous in $\Omega$ and is positive on $\partial\Omega$. Moreover,
    \begin{equation*}
        -\Delta\overline{w}(x)=(\varepsilon-\varepsilon^2)H(x)^{\varepsilon-2}|\nabla H(x)|^{2}.
    \end{equation*}
    We claim that $|\nabla H(x)|\gtrsim|x|^{\phi-1}$. Indeed, decomposing $\nabla H(x)$ into its radial and angular part:
    \begin{equation*}
        \nabla H(x)=\nabla_{rad}H(x)+\nabla_{ang}H(x):=\frac{\phi H(x)}{|x|}\cdot\frac{x}{|x|}+\Big(\nabla H(x)-\frac{\phi H(x)}{|x|}\cdot\frac{x}{|x|}\Big).
    \end{equation*}
    Recall that $H(x)=|x|^{\phi}E(\frac{x}{|x|})$, where $E(\vec{\theta})$ is the first eigenfunction of the Laplace-Beltrami operator $\Delta_{\Sigma}=\Delta_{\mathbb{S}^{n-1}}$. Then, as $\partial\Sigma$ is $C^{1,1}$, we have
    \begin{equation*}
        |\nabla H(x)|\geq|\nabla_{ang}H(x)|=|x|^{\phi-1}|\nabla E(\frac{x}{|x|})|\gtrsim|x|^{\phi-1}
    \end{equation*}
    when $dist(\frac{x}{|x|},\partial\Sigma)$ is sufficiently small. 
    
    On the other hand, When $dist(\frac{x}{|x|},\partial\Sigma)$ is bounded away from $0$, we have 
    \begin{equation*}
        |\nabla H(x)|\geq|\nabla_{rad}H(x)|=\frac{\phi H(x)}{|x|}=\phi|x|^{\phi-1}E(\frac{x}{|x|})\gtrsim|x|^{\phi-1}.
    \end{equation*}
    Therefore, recalling  $\phi=\frac{2}{1+\gamma}$, it follow
    \begin{equation*}
        -\Delta\overline{w}(x)\gtrsim\Big(|x|^{\phi}\cdot dist(\frac{x}{|x|},\partial\Sigma)\Big)^{\varepsilon-2}\cdot|x|^{2(\phi-1)}\sim|x|^{\frac{2\varepsilon}{1+\gamma}-2}\cdot\eta^{\varepsilon-2}(x).
    \end{equation*}
    Since the right-hand side of \eqref{eq. growth rate special condition in critical case} has the following growth:
    \begin{equation*}
        H(x)^{-\gamma}\sim|x|^{-\frac{2\gamma}{1+\gamma}}\cdot\eta^{-\gamma}(x).
    \end{equation*}
    As we have chosen $\varepsilon=1-\frac{\gamma}{2}$, we see $\frac{2\varepsilon}{1+\gamma}-2<-\frac{2\gamma}{1+\gamma}$ and $\varepsilon-2<-\gamma$, so
    \begin{equation*}
        -\Delta\overline{w}(x)\gtrsim H^{-\gamma}(x)\quad\mbox{for }x\in\Omega.
    \end{equation*}
    Then, by choosing  $K$ sufficiently large, $K\cdot\overline{w}$ is a continuous super-solution of \eqref{eq. growth rate special condition in critical case}, therefore the problem \eqref{eq. growth rate special condition in critical case} is solvable. This finishes the proof of $(a)\Rightarrow(b)$.
\end{proof}

Finally, we prove Theorem~\ref{thm. ratio tends to 1}.
\begin{proof}[Proof of Theorem~\ref{thm. ratio tends to 1}]
    From \cite[Theorem 1.4]{GuLiZh25}, we know $u/v$ is bounded away from $0$ and $\infty$. The pointwise continuity of the ratio $u/v$ in the interior of $\Omega$ is obvious, while its continuity near the boundary is more difficult to prove. We let $\widetilde{u}\geq\widetilde{v}$ be the replacements of $u$ and $v$, so that $\widetilde{u}$ solves
    \begin{equation*}
        \left\{\begin{aligned}
            &-\Delta\widetilde{u}=\widetilde{u}^{-\gamma}&\mbox{in }&\Omega ,\\
            &\widetilde{u}=\max\{u,v\}&\mbox{in }&\partial\Omega,
        \end{aligned}\right.
    \end{equation*}
    and $\widetilde{v}$ solves
    \begin{equation*}
        \left\{\begin{aligned}
            &-\Delta\widetilde{v}=\widetilde{v}^{-\gamma}&\mbox{in }&\Omega,\\
            &\widetilde{v}=\min\{u,v\}&\mbox{in }&\partial\Omega.
        \end{aligned}\right.
    \end{equation*}
    It then suffices to show $\widetilde{u}/\widetilde{v}$ tends to $1$ near the boundary. Let $y\in B_{1/16}\cap\partial Cone_{\Sigma}$ and $y\neq0$, then it follows from \cite[Theorem 1.6 (a)]{GuLiZh25} that
    \begin{equation*}
        \lim_{x\to y}\frac{\widetilde{u}(x)}{\widetilde{v}(x)}=1.
    \end{equation*}
    It remains to show
    \begin{equation}\label{eq. u/v(0)=1}
        \lim_{x\to0}\frac{\widetilde{u}(x)}{\widetilde{v}(x)}=1.
    \end{equation}
    Let
    \begin{equation*}
        w=\widetilde{u}-\widetilde{v}\geq0.
    \end{equation*}
    Then, we see $\Delta w\geq0$ because
    \begin{equation*}
        \Delta w=\Delta\widetilde{u}-\Delta\widetilde{v}=\frac{-1}{\widetilde{u}^{\gamma}}-\frac{-1}{\widetilde{v}^{\gamma}}\geq0.
    \end{equation*}
    By Lemma~\ref{lem. classical BHP}, we see $w(x)\lesssim H(x)\sim|x|^{\phi}\eta(x)$. On the other hand, we see from Theorem~\ref{thm. modulus of continuity} (1) that $\widetilde{v}\geq c\cdot U(x)$. From Theorem~\ref{thm. asymptotic estimate}, we see
    \begin{equation*}
        \lim_{x\to0}\frac{H(x)}{U(x)}=0.
    \end{equation*}
    Therefore, when $x\to0$,
    \begin{equation*}
        1\leq\frac{\widetilde{u}(x)}{\widetilde{v}(x)}\leq1+\frac{w(x)}{\widetilde{v}(x)}=1+O(\frac{H(x)}{U(x)})=1+o(1).
    \end{equation*}
    This proves \eqref{eq. u/v(0)=1}.
\end{proof}

\appendix
\section{Notations and conventions}\label{sec. appendix: notations}
We clarify some notations, conventions, and definitions in Appendix~\ref{sec. appendix: notations}.

Throughout this paper, we assume that  $\Sigma\subseteq\partial B_{1}$ satisfies that the pair $(\Sigma,\gamma)$ is critical, which means  the ``frequency'' $\phi=\phi_{\Sigma}$ (see Definition~\ref{def. frequency}) of $Cone_{\Sigma}$ satisfies
\begin{equation*}
    \phi=\frac{2}{1+\gamma}.
\end{equation*}
\begin{defn}
    As was already mentioned in the main results, in this paper, we denote
    \begin{equation*}
        \Omega=Cone_{\Sigma}\cap B_{1},
    \end{equation*}
    and we define $U(x)$ as a special solution to \eqref{eq. main} in $\Omega$, which vanishes on $\partial\Omega$.
\end{defn}

\begin{defn}[Smooth cones]\label{def. a good cone}
    A cone $Cone_{\Sigma}$ is called a $C^{1,1}$ epigraphical cone in $\mathbb{R}^{n}$ ($n\geq2$)  if it can be represented as the epigraph of a function $g(x')$, that is
    \begin{equation*}
        Cone_{\Sigma}=\{x_{n}>g(x')\},
    \end{equation*}
    where the function $g(x')$ satisfies:
    \begin{itemize}
        \item[(1)] For every $\lambda\geq0$ and $x'\in\mathbb{R}^{n-1}$, $g(\lambda\cdot x')=\lambda\cdot g(x')$;
        \item[(2)] $g(x')$ is a $C^{1,1}$ function in the annulus $B_{2}'\setminus B_{1}'$.
    \end{itemize}
    We denote the $C^{0,1}$ norm and $C^{1,1}$ semi-norm of $Cone_{\Sigma}$ as
    \begin{align*}
        L=\|Cone_{\Sigma}\|_{C^{0,1}}:=&\|g(x')\|_{C^{0,1}(B_{2}'\setminus B_{1}')},\\
        K=[Cone_{\Sigma}]_{C^{1,1}}:=&[g(x')]_{C^{1,1}(B_{2}'\setminus B_{1}')}.
    \end{align*}
\end{defn}

In this paper, we adopt the following conventions throughout our analysis and computations to denote the equivalence of quantities.

\begin{defn}[Uniform constants and equivalence relation]\label{def. convention of uniform and equivalence}
    A constant $C>0$ is called uniform if it depends at most on $(n,\gamma,L,K)$. Here, $(L,K)$ represents the norms of $Cone_{\Sigma}$ as mentioned in Definition~\ref{def. a good cone}. Furthermore, for two non-negative quantities $q_{1}$ and $q_{2}$, we adopt the following notations for simplicity:
    \begin{itemize}
        \item $q_{1}\lesssim q_{2}$, if $q_{1}\leq C\cdot q_{2}$ for some uniform constant $C$;
        \item $q_{1}\gtrsim q_{2}$, if $q_{1}\geq C^{-1}\cdot q_{2}$ for some uniform constant $C$;
        \item $q_{1}\sim q_{2}$, if $q_{1}\lesssim q_{2}$ and $q_{1}\gtrsim q_{2}$ both hold.
    \end{itemize}
\end{defn}
\begin{remark}
    We need to emphasize the danger in the transitivity of the equivalence relation $\sim$. Transitivity holds only when there is a uniformly finite chain of equivalence. For instance, consider  the following chain of equivalences:   
    \begin{equation*}
        a_{1}\sim a_{2},\ a_{2}\sim a_{3},\ a_{3}\sim a_{4},\ \cdots.
    \end{equation*}
    Then we can say $a_{1}\sim a_{100}$, but we should not say $a_{1}\sim a_{k}$ for an arbitrarily large $k$.
\end{remark}
We also require a defining function $\eta(x)$ for each $x\in Cone_{\Sigma}$, or equivalently $\eta_{\Sigma}(\theta)$ for each $\theta\in\Sigma$. 
The proof follows from a partition-of-unity argument, which we omit here for brevity.
\begin{lemma}[The defining function of $\Sigma$]\label{lem. definiting function in a patch}
    Let $Cone_{\Sigma}$ be a $C^{1,1}$ epigraphical cone with norm (L,K), as mentioned in Definition~\ref{def. a good cone}. Then, the region $\Sigma$ is $C^{1,1}$ and admits a defining function $\eta_{\Sigma}(\vec{\theta})\in C^{1,1}(\Sigma)$. The defining function $\eta_{\Sigma}(\theta)$ satisfies the following properties:
    \begin{itemize}
        \item[(a)] $\eta_{\Sigma}(\vec{\theta})>0$ for every interior point $\vec{\theta}\in\Sigma$, and $\eta_{\Sigma}(\vec{\theta})=0$ when $\vec{\theta}\in\partial\Sigma$;
        \item[(b)] For every $\vec{\theta}\in\Sigma$ such that $dist(\vec{\theta},\partial\Sigma)\geq d_{0}(n,L,K)$, we have $\eta_{\Sigma}(\vec{\theta})\geq h_{0}(n,L,K)$, here $d_{0}$ and $h_{0}$ are two sufficiently small but uniform constants;
        \item[(c)] For every $\vec{\theta}\in\Sigma$ such that $dist(\vec{\theta},\partial\Sigma)\leq d_{0}(n,L,K)$, we have $\eta_{\Sigma}(\vec{\theta})\sim dist(\vec{\theta},\partial\Sigma)$;
        \item[(d)] $\|\eta_{\Sigma}\|_{C^{1,1}(\Sigma)}\sim1$ and $|\nabla\eta_{\Sigma}(\vec{\theta})|\sim1$ for $\theta\in\partial\Sigma$;
        \item[(e)] For $0\leq h<h_{0}(n,L,K)$, the level set $\eta_{\Sigma}^{-1}(h)$ is a $C^{1,1}$ hypersurface in $\partial B_{1}$ satisfying
        \begin{eqnarray*}
            \mathcal{H}^{n-2}\Big(\eta_{\Sigma}^{-1}(h)\Big)\sim\mathcal{H}^{n-2}(\partial\Sigma),
        \end{eqnarray*}
        and $|\nabla\eta(\theta)|\sim1$ everywhere on $\eta_{\Sigma}^{-1}(h)$;
        \item[(f)] For $0\leq h<h_{0}(n,L,K)$, the sub-level set $\eta_{\Sigma}^{-1}([0,h])$ has area proportional to $h$, i.e.:
        \begin{equation*}
            \mathcal{H}^{n-1}\Big(\eta_{\Sigma}^{-1}([0,h])\Big)\sim h\cdot\mathcal{H}^{n-1}(\partial\Sigma).
        \end{equation*}
    \end{itemize}
    \begin{defn}[The defining function of $Cone_{\Sigma}$]\label{def. defining function eta}
        For the defining function $\eta_{\Sigma}(\vec{\theta})$ mentioned in Lemma~\ref{lem. definiting function in a patch}, we define $\eta(x):Cone_{\Sigma}\to\mathbb{R}$ as
        \begin{equation*}
            \eta(x):=\eta_{\Sigma}(\frac{x}{|x|}).
        \end{equation*}
    \end{defn}
\end{lemma}
\begin{exmp}\label{ex. defining function n=2}
    For a critical cone in $\mathbb{R}^{2}$, given in Example~\ref{ex. 2d how large is the critical angle}, a suitable choice of the defining function is
    \begin{equation*}
        \eta(x)=\cos{(\phi\theta)}.
    \end{equation*}
\end{exmp}

\section{Review of useful facts}\label{sec. appendix: review}
In Appendix~\ref{sec. appendix: review}, we review several important facts that will be  frequently used throughout  this paper.

We first define some important cylindrical domains.
\begin{defn}[Cylindrical domains]
    Let $\Gamma=\{x_{n}=g(x')\}$ be a graph of a Lipschitz function $g$. For a boundary point $x=(x',g(x'))\in\Gamma$, we define
    \begin{itemize}
        \item a grounded cylinder $\mathcal{GC}_{r}(x)$ as
        \begin{equation*}
            \mathcal{GC}_{r}(x):=\{y=(y',y_{n}):\quad|x'-y'|\leq r,\quad0\leq y_{n}-g(y')\leq r\};
        \end{equation*}
        \item a suspended cylinder $\mathcal{SC}_{r,\delta}(x)$ as
        \begin{equation*}
            \mathcal{SC}_{r,\delta}(x):=\{y=(y',y_{n}):\quad|x'-y'|\leq r,\quad\delta r\leq y_{n}-g(y')\leq r\}.
        \end{equation*}
    \end{itemize}
\end{defn}

We next recall the boundary Harnack principle obtained by Kemper in \cite{Ke72}.
\begin{lemma}[Classical boundary Harnack principle]\label{lem. classical BHP}
    Suppose that $g(x'):\mathbb{R}^{n-1}\to\mathbb{R}$ is a Lipschitz function near the origin, with $g(0)=0$ and $\|g\|_{C^{0,1}(B_{1}')}\leq L$. Assume that $u,v\geq0$ are harmonic in $\mathcal{GC}_{1}(0)$, and that $u,v$ both vanish on the curve $\{x_{n}=g(x')\}$ in the trace or limit sense. Then
    \begin{equation*}
        \frac{u(x)}{v(x)}\sim\frac{u(\frac{1}{2}\vec{e_{n}})}{v(\frac{1}{2}\vec{e_{n}})}\mbox{ in }\mathcal{GC}_{1/2}(0).
    \end{equation*}
\end{lemma}
\begin{proof}
    See \cite{Ke72}.
\end{proof}

When studying the singular Lane-Emden-Fowler equation, several useful facts are presented below.
\begin{lemma}[Maximum principle]
    If $u$ is a super-solution and $v$ is a sub-solution to the equation \eqref{eq. main} in a bounded open set $\mathcal{D}$, such that $u\geq v$ on $\partial\mathcal{D}$, then $u\geq v$ everywhere in $\mathcal{D}$.
\end{lemma}
\begin{proof}
    Notice that $u^{-\gamma}$ is a decreasing function, so the proof is standard.
\end{proof}
\begin{lemma}[Interior lower bound and Harnack principle]\label{lem. review : interior lower bound}
    Assume that $u\geq0$ is a solution to \eqref{eq. main} in $B_{r}$, then
    \begin{itemize}
        \item[(1)] $\displaystyle \min_{x\in B_{r/2}}u(x)\gtrsim r^{\frac{2}{1+\gamma}}$;
        \item[(2)] $\displaystyle \min_{x\in B_{r/2}}u(x)\sim\max_{x\in B_{r/2}}u(x)$.
    \end{itemize}
\end{lemma}
\begin{proof}
    Without loss of generality, we can assume $r=1$ using the invariant scaling
    \begin{equation*}
        \widetilde{u}(x)=r^{-\frac{2}{1+\gamma}}u(rx),\quad x\in B_{1}.
    \end{equation*}
    The lower bound estimate (1) can be obtained from the following lower barrier
    \begin{equation*}
        v=\Big((2n)^{-1/\gamma}-|x|^{2}\Big)_{+}.
    \end{equation*}
    The interior Harnack principle (2) follows from the standard one for linear equations, by noting that the lower bound estimate of $u$ yields an upper bound of the right-hand side term $u^{-\gamma}$.
\end{proof}
\begin{lemma}[Well-posedness]
    Let $\Omega$ be a bounded domain with a Lipschitz boundary. Let $\varphi\geq0$ be a continuous function on $\partial\Omega$. Then there exists a unique classical solution $u\in C^{\infty}_{loc}(\Omega)\cap C(\overline{\Omega})$ to \eqref{eq. main} with boundary data $\varphi$.
\end{lemma}
\begin{proof}
    See \cite[Theorem 1.1]{GuLiZh25}.
\end{proof}
\begin{lemma}[Another Harnack principle]\label{lem. u is small near the boundary}
    Assume that $u$ satisfies \eqref{eq. main} in $\mathcal{GC}_{2r}(0)$ and vanishes on $\Gamma$, where the boundary $\Gamma$ has $\|\Gamma\|_{C^{0,1}}\leq L$. Then,
    \begin{equation*}
        C^{-1}\max_{\mathcal{GC}_{r}(0)}u\leq u(r\vec{e_{n}})\leq C\min_{\mathcal{SC}_{r,\delta}(0)}u,
    \end{equation*}
    for some $C=C(n,L,\gamma,\delta)$.
\end{lemma}
\begin{proof}
    See \cite[Lemma 3.2]{GuLiZh25}.
\end{proof}
\begin{lemma}[A new boundary Harnack principle]\label{lem. Harnack Comparable, general Lipschitz domain}
    Assume that $[g(x')]_{C^{0,1}}\leq L$ and $g(0)=0$. If $u,v\geq0$ both satisfy \eqref{eq. main} in $\mathcal{GC}_{3R}(0)$ and vanish on $\Gamma$, then in $\mathcal{GC}_{R}(0)$, we have
    \begin{equation*}
        \min\Big\{\frac{\|u\|_{L^{\infty}(\mathcal{GC}_{3R}(0))}}{\|v\|_{L^{\infty}(\mathcal{GC}_{3R}(0))}},1\Big\}\lesssim\frac{u}{v}\lesssim\max\Big\{\frac{\|u\|_{L^{\infty}(\mathcal{GC}_{3R}(0))}}{\|v\|_{L^{\infty}(\mathcal{GC}_{3R}(0))}},1\Big\},
    \end{equation*}
    and
    \begin{equation*}
        \min\Big\{\frac{u(R\vec{e_{n}})}{v(R\vec{e_{n}})},1\Big\}\lesssim\frac{u}{v}\lesssim\max\Big\{\frac{u(R\vec{e_{n}})}{v(R\vec{e_{n}})},1\Big\}.
    \end{equation*}
\end{lemma}
\begin{proof}
    See \cite[Theorem 1.4]{GuLiZh25}.
\end{proof}
\begin{lemma}[Boundary estimate in a semi-$C^{1,1}$ domain for $\gamma>1$]\label{lem. subcritical boundary estimate}
    Assume that every boundary point $p\in\Gamma\cap\{|x'|\leq3\}$ admits an exterior ball with radius $\frac{1}{K}$ below the Lipschitz graph $\Gamma=\{x_{n}=g(x')\}$. Let $u$ be a solution to \eqref{eq. main} in $\mathcal{GC}_{3}(0)$ such that $u=0$ on $\Gamma$ and $u(\vec{e_{n}})\sim1$, then we have in the case $\gamma>1$ that
    \begin{equation*}
        u(x)\sim|x_{n}-g(x')|^{\frac{2}{1+\gamma}}\sim \left(dist(x,\Gamma)\right)^{\frac{2}{1+\gamma}}\quad\mbox{in }\mathcal{GC}_{1}(0).
    \end{equation*}
\end{lemma}
\begin{proof}
    See the proof in \cite{GuLi93}, or refer to  \cite[Theorem 1.2 (a)(d)]{GuLiZh25} and \cite[Corollary 1.1 (a)]{GuLiZh25}.
\end{proof}
Finally, we state a straightforward corollary of Lemma~\ref{lem. Harnack Comparable, general Lipschitz domain}.
\begin{cor}[One point control]\label{cor. nonlinear boundary harnack accurate}
    Assume that $[g(x')]_{C^{0,1}}\leq L$ and $g(0)=0$. If $u,v\geq0$ both satisfy \eqref{eq. main} in $\mathcal{GC}_{3R}(0)$ and vanish on $\Gamma$, and suppose that
    \begin{equation*}
        u(R\vec{e_{n}})\sim v(R\vec{e_{n}}),
    \end{equation*}
    then
    \begin{equation*}
        u\sim v\mbox{ in }\mathcal{GC}_{R}(0).
    \end{equation*}
\end{cor}

\noindent \textbf{Acknowledgments} 
L. Wu is partially supported by National Natural Science Foundation of China (Grant No. 12401133) and the Guangdong Basic and Applied Basic Research Foundation (2025B151502069).

\vspace{2mm}

\noindent \textbf{Conflict of interest.} The authors do not have any possible conflicts of interest.

\vspace{2mm}

\noindent \textbf{Data availability statement.}
 Data sharing is not applicable to this article, as no data sets were generated or analyzed during the current study.

\bibliographystyle{abbrv} 
\parskip=0pt
\small

\begin{thebibliography}{10}

\bibitem{BoOr10}
L.~Boccardo and L.~Orsina.
\newblock Semilinear elliptic equations with singular nonlinearities.
\newblock {\em Calc. Var. Partial Differential Equations}, 37(3-4):363--380, 2010.

\bibitem{Bo00}
K.~Bogdan.
\newblock Sharp estimates for the green function in lipschitz domains.
\newblock {\em J. Math. Anal. Appl.}, 243(2):326--337, 2000.

\bibitem{CrRaTa77}
M.~G. Crandall, P.~H. Rabinowitz, and L.~Tartar.
\newblock On a {D}irichlet problem with a singular nonlinearity.
\newblock {\em Comm. Partial Differential Equations}, 2(2):193--222, 1977.

\bibitem{dP92}
M.~A. del Pino.
\newblock A global estimate for the gradient in a singular elliptic boundary value problem.
\newblock {\em Proc. Roy. Soc. Edinburgh Sect. A}, 122(3-4):341--352, 1992.

\bibitem{DuWa18}
J.~Duraj and V.~Wachtel.
\newblock Green function of a random walk in a cone.
\newblock {\em arXiv preprint arXiv:1807.07360}, 2018.

\bibitem{FuMa60}
W.~Fulks and J.~S. Maybee.
\newblock A singular non-linear equation.
\newblock {\em Osaka Math. J.}, 12:1--19, 1960.

\bibitem{Go86}
S.~M. Gomes.
\newblock On a singular nonlinear elliptic problem.
\newblock {\em SIAM Journal on Mathematical Analysis}, 17(6):1359--1369, 1986.

\bibitem{GuLi93}
C.~Gui and F.~Lin.
\newblock Regularity of an elliptic problem with a singular nonlinearity.
\newblock {\em Proc. Roy. Soc. Edinburgh Sect. A}, 123(6):1021--1029, 1993.

\bibitem{GuLiZh25}
Y.~Guo, C.~Li, and C.~Zhang.
\newblock Boundary regularity theory of the singular lane-emden-fowler equation in a lipschitz domain.
\newblock {\em arXiv preprint arXiv:2503.16095}, 2025.

\bibitem{GuWu25}
Y.~Guo and L.~Wu.
\newblock Classification of solutions to a singular fractional problem in the half space, 2025, preprint.

\bibitem{HuZh24}
Y.~Huang and C.~Zhang.
\newblock Optimal h\"older convergence of a class of singular steady states to the bahouri–chemin patch.
\newblock {\em Calc. Var. Partial Differential Equations}, 64(6):Paper No. 191, 31, 2025.

\bibitem{Ke72}
J.~T. Kemper.
\newblock A boundary harnack principle for lipschitz domains and the principle of positive singularities.
\newblock {\em Comm. Pure Appl. Math.}, 25:247--255, 1972.

\bibitem{LaMc91}
A.~C. Lazer and P.~J. McKenna.
\newblock On a singular nonlinear elliptic boundary-value problem.
\newblock {\em Proc. Amer. Math. Soc.}, 111(3):721--730, 1991.

\bibitem{MoMuSc24a}
L.~Montoro, L.~Muglia, and B.~Sciunzi.
\newblock The classification of all weak solutions to {$-\Delta u=u^{-\gamma}$} in the half-space.
\newblock {\em aeXiv preprint arXiv:2404.03343}, 2024.

\bibitem{MoMuSc24b}
L.~Montoro, L.~Muglia, and B.~Sciunzi.
\newblock Classification of solutions to {$-\Delta u=u^{-\gamma}$} in the half-space.
\newblock {\em Math. Ann.}, 389(3):3163--3179, 2024.

\bibitem{OlPe18}
F.~Oliva and F.~Petitta.
\newblock Finite and infinite energy solutions of singular elliptic problems: existence and uniqueness.
\newblock {\em J. Differential Equations}, 264(1):311--340, 2018.

\bibitem{St76}
C.~Stuart.
\newblock Existence and approximation of solutions of non-linear elliptic equations.
\newblock {\em Math. Z.}, 147(1):53--63, 1976.

\end{thebibliography}

\vspace{4mm}
\newpage

\noindent(L. Wu)\par\nopagebreak
\noindent\textsc{School of Mathematics, South China University of Technology, Guangzhou, 510640, P. R. China}

\noindent Email address: {\tt leyunwu@scut.edu.cn}
\vspace{3ex}

\noindent(C. Zhang)\par\nopagebreak
\noindent\textsc{School of Mathematical Sciences, Fudan University, Shanghai 200433, P. R. China}

\noindent Email address: {\tt zhangchilin@fudan.edu.cn}

\end{document}